\documentclass[11pt,reqno,a4paper]{amsart}
\usepackage{amssymb,amsmath}
\usepackage{amscd}
\usepackage{pspicture}
\usepackage{graphicx}
\usepackage{color}
\usepackage{mathptmx}

\usepackage[matrix,arrow,curve,frame]{xy}    

\xymatrixcolsep{1.9pc}                          
\xymatrixrowsep{1.9pc}
\newdir{ >}{{}*!/-5pt/\dir{>}}                  

\makeatletter

\renewcommand{\phi}{\varphi}
\renewcommand{\leq}{\leqslant}
\renewcommand{\geq}{\geqslant}
\renewcommand{\epsilon}{\varepsilon}

\renewcommand{\kappa}{\varkappa}

 \DeclareMathOperator{\ind}{ind}
\DeclareMathOperator{\cyl}{cyl}

\DeclareMathOperator{\sd}{sd} 
 \DeclareMathOperator{\Map}{Map}

 \DeclareMathOperator{\cone}{cone}
\DeclareMathOperator{\Hom}{Hom} 
 
 \DeclareMathOperator{\id}{id}

\DeclareMathOperator{\Alg}{Alg} \DeclareMathOperator{\colim}{colim}

 \DeclareMathOperator{\kr}{Ker}

 \DeclareMathOperator{\Mod}{Mod}
 \DeclareMathOperator{\Ob}{Ob}
\DeclareMathOperator{\CAlg}{CAlg}

\newcommand{\sn}{\par\smallskip\noindent}

\newcommand {\lp}{\colim}

\newcommand{\lra}[1]{\bl{#1}\longrightarrow\relax}
\newcommand{\bl}[1]{\buildrel #1\over}
\newcommand{\shs}{SH_{S^1}}
\newcommand{\shsg}{SH_{S^1,\bb G}}
\newcommand{\cc}{\mathcal}
\newcommand{\bb}{\mathbb}

\newcommand{\ff}{\mathfrak}

\newcommand{\op}{{\textrm{\rm op}}}

\newcommand{\wt}{\widetilde}

\newcommand{\ahaw}{{\Alg_{k}}}

\newcommand{\cahaw}{{\CAlg}_k}

\newcommand{\inda}{{\Alg_k^{\ind}}}

\newtheorem{thm}{Theorem}[section]
\newtheorem{prop}[thm]{Proposition}
\newtheorem{cor}[thm]{Corollary}
\newtheorem{lem}[thm]{Lemma}

\newtheorem*{questions}{Questions}
\newtheorem*{rem}{Remark}

\newtheorem*{exs}{Examples}

\newtheorem*{defs}{Definition}

\makeatother

\begin{document}

\footskip30pt


\title{Algebraic Kasparov K-theory. II}
\author{Grigory Garkusha}
\address{Department of Mathematics, Swansea University, Singleton Park, Swansea SA2 8PP, UK}
\email{G.Garkusha@swansea.ac.uk}

\urladdr{http://math.swansea.ac.uk/staff/gg/}



\begin{abstract}
A kind of motivic stable homotopy theory of algebras is developed.
Explicit fibrant replacements for the $S^1$-spectrum and $(S^1,\bb
G)$-bispectrum of an algebra are constructed. As an application,
unstable, Morita stable and stable universal bivariant theories are
recovered. These are shown to be embedded by means of contravariant
equivalences as full triangulated subcategories of compact
generators of some compactly generated triangulated categories.
Another application is the introduction and study of the symmetric
monoidal compactly generated triangulated category of $K$-motives.
It is established that the triangulated category $kk$ of
Corti\~{n}as--Thom~\cite{CT} can be identified with the $K$-motives
of algebras. It is proved that the triangulated category of
$K$-motives is a localization of the triangulated category of
$(S^1,\bb G)$-bispectra. Also, explicit fibrant $(S^1,\bb
G)$-bispectra representing stable algebraic Kasparov $K$-theory and
algebraic homotopy $K$-theory are constructed.
\end{abstract}
\maketitle

\thispagestyle{empty} \pagestyle{plain}

\newdir{ >}{{}*!/-6pt/@{>}} 


\section{Introduction}

Throughout the paper $k$ is a fixed commutative ring with unit and
$\ahaw$ is the category of non-unital $k$-algebras and non-unital
$k$-homomorphisms. Also, $F$ is a fixed field and $Sm/F$ is the category
of smooth algebraic varieties over $F$. If $\cc C$ is a category and $A,B$
are objects of $\cc C$, we shall often write $\cc C(A,B)$ to denote
the Hom-set $\Hom_{\cc C}(A,B)$.

$\bb A^1$-homotopy theory is the homotopy theory of motivic spaces,
i.e. presheaves of simplicial sets defined on $Sm/F$
(see~\cite{MV,VoeICM}). Each object $X\in Sm/F$ is regarded as the
motivic space $\Hom_{Sm/F}(-,X)$. The affine line $\bb A^1$ plays
the role of the interval.

$k[t]$-homotopy theory is the homotopy theory of simplicial functors
defined on non-unital algebras, where each algebra $A$ is regarded
contravariantly as the space $rA=\Hom_{\ahaw}(A,-)$ so that we can
study algebras from a homotopy theoretic viewpoint
(see~\cite{Gar,Gark}). The role of the interval is played by the
space $r(k[t])$ represented by the polynomial algebra $k[t]$. This
theory borrows methods and approaches from $\bb A^1$-homotopy
theory. Another source of ideas and techniques for $k[t]$-homotopy
theory originates in Kasparov $K$-theory of $C^*$-algebras.

In~\cite{Gar} a kind of unstable motivic homotopy theory of algebras
was developed. In order to develop stable motivic homotopy theory of
algebras and -- most importantly -- to make the explicit computations
of fibrant replacements for suspension spectra $\Sigma^\infty rA$, $A\in\ahaw$,
presented in this paper, one first needs to introduce and study
``unstable, Morita stable and stable Kasparov $K$-theory spectra"
$\bb K(A,B)$, $\bb K^{mor}(A,B)$ and $\bb K^{st}(A,B)$ respectively,
where $A,B$ are algebras. We refer the reader to~\cite{Gark} for
properties of the spectra. The aim of this paper is to develop
stable motivic homotopy theory of algebras.

Throughout we work with a certain small subcategory $\Re$ of $\ahaw$
and the category $U_\bullet\Re$ of certain pointed simplicial
functors on $\Re$. $U_\bullet\Re$ comes equipped with a motivic
model structure. We write $Sp(\Re)$ to denote the stable model category of
$S^1$-spectra associated with the model category $U_\bullet\Re$.
$\bb K(A,-)$, $\bb K^{mor}(A,-)$ and $\bb K^{st}(A,-)$ are examples
of fibrant $\Omega$-spectra in $Sp(\Re)$ (see~\cite{Gark}).

One of the main results of the paper says that $\bb K(A,-)$ is a
fibrant replacement for the suspension spectrum $\Sigma^\infty rA\in Sp(\Re)$
of an algebra $A\in\Re$. Namely, there is a natural weak equivalence of
spectra
   $$\Sigma^\infty rA\lra{}\bb K(A,-)$$
in $Sp(\Re)$ (see Theorem~\ref{mainres}).

This is an analog of a similar result by the author and
Panin~\cite{GP3} computing a fibrant replacement of the suspension
$\bb P^1$-spectrum $\Sigma^\infty_{\bb P^1} X_+$ of a smooth
algebraic variety $X$. The main reason that computation of a fibrant
replacement for $\Sigma^\infty_{\bb P^1} X_+$ is possible is the
existence of framed correspondences of Voevodsky~\cite{Voe2} on
homotopy groups of (motivically fibrant) $\bb P^1$-spectra. In turn,
the main reason why the computation of a fibrant replacement for
$\Sigma^\infty rA$ is possible is that algebras have universal
extensions.

Let $\shs(\Re)$ denote the homotopy category of $Sp(\Re)$.
$\shs(\Re)$ plays the same role as the stable homotopy category of
motivic $S^1$-spectra $\shs(F)$ over a field $F$. It is a compactly
generated triangulated category with compact generators
$\{\Sigma^\infty rA[n]\}_{A\in\Re,n\in\bb Z}$. One of the important
consequences of the above computation is that we are able to give an
{\it explicit\/} description of the Hom-groups
$\shs(\Re)(\Sigma^\infty rB[n],\Sigma^\infty rA)$. Precisely, there
is an isomorphism of abelian groups (see Corollary~\ref{rbra})
   $$\shs(\Re)(\Sigma^\infty rB[n],\Sigma^\infty rA)\cong\bb K_n(A,B),\quad A,B\in\Re,n\in\bb Z.$$
It is important to note that the groups $\bb K_n(A,B)$ have an {\it explicit\/}
description in terms of non-unital algebra homomorphisms (see~\cite[Section~7]{Gark} for details).

We also show in Theorem~\ref{embed} that the full subcategory
$\cc S$ of $\shs(\Re)$ spanned by the compact generators
$\{\Sigma^\infty rA[n]\}_{A\in\Re,n\in\bb Z}$ is triangulated and
there is a contravariant equivalence of triangulated categories
   $$D(\Re,\ff F)\lra{\sim}\cc S$$
with $\Re\to D(\Re,\ff F)$ the universal unstable excisive
homotopy invariant homology theory in the sense of~\cite{Gar1} with
respect to the class of $k$-split surjective algebra homomorphisms
$\ff F$. This equivalence is an extension of the contravariant
functor $A\in\Re\mapsto\Sigma^\infty rA\in\shs(\Re)$ to $D(\Re,\ff
F)$. Thus $D(\Re,\ff F)$ is recovered from $\shs(\Re)$. It also
follows that the {\it small\/} triangulated category $D(\Re,\ff
F)^{\op}$ lives inside the ``big" ambient triangulated category
$\shs(\Re)$. This is reminiscent of Voevodsky's Theorem~\cite{Voe1}
saying that there is a full embedding of the {\it small\/}
triangulated category $DM^{eff}_{gm}(F)$ of effective geometrical
motives into the ``big" triangulated category $DM^{eff}(F)$ of
motivic complexes of Nisnevich sheaves with transfers.

Next, we introduce matrices into the game. Namely, if we
localize $\shs(\Re)$ with respect to the family of compact objects
   $$\{\cone(\Sigma^\infty r(M_nA)\to\Sigma^\infty rA)\}_{n>0}$$
we shall get a compactly generated triangulated category
$\shs^{mor}(\Re)$ with compact generators $\{\Sigma^\infty
rA[n]\}_{A\in\Re,n\in\bb Z}$. It is in fact the homotopy category of
a model category $Sp_{mor}(\Re)$, which is the same category as $Sp(\Re)$
but with a new model structure. We construct in a similar way a
compactly generated triangulated category $\shs^{\infty}(\Re)$,
obtained from $\shs(\Re)$ by localization with respect to the family
of compact objects
   $$\{\cone(\Sigma^\infty r(M_\infty A)\to\Sigma^\infty rA)\},$$
where $M_\infty A=\cup_nM_n A$. It is also the homotopy category of
a model category $Sp_{\infty}(\Re)$, which is the same category as $Sp(\Re)$
but with a new model structure.

We prove in Theorems~\ref{mainresmor} and~\ref{mainresst} that for any
algebra $A\in\Re$ there are natural weak equivalences of
spectra
   $$\Sigma^\infty rA\lra{}\bb K^{mor}(A,-)$$
and
   $$\Sigma^\infty rA\lra{}\bb K^{st}(A,-)$$
in $Sp_{mor}(\Re)$ and $Sp_{\infty}(\Re)$ respectively.
Also, for all $A,B\in\Re$ and
$n\in\bb Z$ there are isomorphisms of abelian groups
   $$\shs^{mor}(\Re)(\Sigma^\infty rB[n],\Sigma^\infty rA)\cong\bb K^{mor}_n(A,B)$$
and
   $$\shs^{\infty}(\Re)(\Sigma^\infty rB[n],\Sigma^\infty rA)\cong\bb K^{st}_n(A,B)$$
respectively. Furthermore, the full
subcategories $\cc S_{mor}$ and $\cc S_{\infty}$ of
$\shs^{mor}(\Re)$ and $\shs^{\infty}(\Re)$ spanned by the compact
generators $\{\Sigma^\infty rA[n]\}_{A\in\Re,n\in\bb Z}$ are
triangulated and there are contravariant equivalences of
triangulated categories
   $$D_{mor}(\Re,\ff F)\lra{\sim}\cc S_{mor}$$
and
   $$D_{st}(\Re,\ff F)\lra{\sim}\cc S_{\infty}.$$
Here $\Re\to D_{mor}(\Re,\ff F)$ (respectively $\Re\to
D_{st}(\Re,\ff F)$) is the universal Morita stable (respectively
stable) excisive homotopy invariant homology theory in the sense
of~\cite{Gar1}. Thus $D_{mor}(\Re,\ff F)$ and $D_{st}(\Re,\ff F)$
are recovered from $\shs^{mor}(\Re)$ and $\shs^{st}(\Re)$
respectively. It also follows that the {\it small\/} triangulated
categories $D_{mor}(\Re,\ff F)^{\op}$, $D_{st}(\Re,\ff F)^{\op}$
live inside the ambient triangulated categories $\shs^{mor}(\Re)$
and $\shs^{st}(\Re)$.

We next introduce a symmetric monoidal compactly generated
triangulated category of $K$-motives $DK(\Re)$ together with a
canonical contravariant functor
   $$M_K:\Re\to DK(\Re).$$
The category $DK(\Re)$ is an analog of the triangulated category of
$K$-motives for algebraic varieties introduced in~\cite{GP,GP1}.

For any algebra $A\in\Re$ its $K$-motive is, by definition, the
object $M_K(A)$ of $DK(\Re)$. We have that
   $$M_K(A)\otimes M_K(B)\cong M_K(A\otimes B)$$
for all $A,B\in\Re$ (see Proposition~\ref{kmo}).

We prove in Theorem~\ref{kmotives} that for any two algebras $A,B\in\Re$ and any integer $n$
there is a natural isomorphism
   \begin{equation*}
    DK(\Re)(M_K(B)[n],M_K(A))\cong\bb K^{st}_n(A,B).
   \end{equation*}
Moreover, if $\cc T$ is the full subcategory of $DK(\Re)$ spanned by
$K$-motives of algebras $\{M_K(A)\}_{A\in\Re}$ then $\cc T$ is
triangulated and there is an equivalence of triangulated categories
   $$D_{st}(\Re,\ff F)\to\cc T^{\op}$$
sending an algebra $A\in\Re$ to its $K$-motive $M_K(A)$ (see Theorem~\ref{kmotives}).
It is also proved in Corollary~\ref{kak} that for any algebra $A\in\Re$ and any integer $n$ one
has a natural isomorphism
   \begin{equation*}
    DK(\Re)(M_K(A)[n],M_K(k))\cong KH_n(A),
   \end{equation*}
where the right hand side is the $n$-th homotopy $K$-theory group in
the sense of Weibel~\cite{W1}. This result is reminiscent of a
similar result for $K$-motives of algebraic varieties in the sense
of~\cite{GP,GP1} identifying the $K$-motive of the point with
algebraic $K$-theory.

In~\cite{CT} Corti\~nas--Thom constructed a universal excisive
homotopy invariant and $M_\infty$-invariant homology theory on all
$k$-algebras
   $$j:\ahaw\to kk.$$
The triangulated category $kk$ is an analog of Cuntz's triangulated
category $kk^{lca}$ whose objects are the locally convex
algebras~\cite{Cu2,Cu,Cu1}.

We show in Theorem~\ref{kktop} that, if we denote by $kk(\Re)$ the
full subcategory of $kk$ spanned by the objects from $\Re$ and
assume that the cone ring $\Gamma k$ in the sense of~\cite{KV} is in
$\Re$, then there is a natural triangulated equivalence
   $$kk(\Re)\lra{\sim}\cc T^{\op}$$
sending $A\in kk(\Re)$ to its $K$-motive $M_K(A)$. Thus we can
identify $kk(\Re)$ with the $K$-motives of algebras. It also follows
that the {\it small\/} triangulated category $kk(\Re)^{\op}$ lives
inside the ambient triangulated category $DK(\Re)$.

One of the equivalent approaches to stable motivic homotopy theory
in the sense of Morel--Voevodsky~\cite{MV} is the theory of
$(S^1,\bb G_m)$-bispectra. The role of $\bb G_m$ in our context is
played by the representable functor $\bb G:=r(\sigma)$, where
$\sigma=(t-1)k[t^{\pm 1}]$. We develop the motivic theory of
$(S^1,\bb G)$-bispectra. As usual they form a model category which
we denote by $Sp_{\bb G}(\Re)$. The homotopy category $SH_{S^1,\bb
G}(\Re)$ of $Sp_{\bb G}(\Re)$ plays the same role as the stable
motivic homotopy category $SH(F)$ over a field $F$. We construct an
explicit fibrant $(S^1,\bb G)$-bispectrum $\Theta^\infty_{\bb
G}\bb{KG}(A,-)$, obtained from fibrant $S^1$-spectra $\bb
K(\sigma^nA,-)$, $n\geq 0$, by stabilization in the
$\sigma$-direction.

The main computational result for bispectra (Theorem~\ref{maingsp}) says that
$\Theta^\infty_{\bb G}\bb{KG}(A,-)$ is a fibrant replacement of the
suspension bispectrum associated with an algebra $A$. Namely, there is a natural weak
equivalence of bispectra in $Sp_{\bb G}(\Re)$
   $$\Sigma_{\bb G}^\infty\Sigma^\infty rA\to\Theta^\infty_{\bb G}\bb{KG}(A,-),$$
where $\Sigma_{\bb G}^\infty\Sigma^\infty rA$ is the suspension
bispectrum of $rA$.

Let $k$ be the field of complex numbers $\bb C$ and let $\cc
K^\sigma(A,-)$ be the (0,0)-space of the bispectrum
$\Theta^\infty_{\bb G}\bb{KG}(A,-)$. We raise a question whether
there is a category $\Re$ of commutative $\bb C$-algebras such that
the fibrant simplicial set $\cc K^\sigma(\bb C,\bb C)$ has the
homotopy type of $\Omega^\infty\Sigma^\infty S^0$. The question is
justified by a recent result of Levine~\cite{Lev} saying that over
an algebraically closed field $F$ of characteristic zero the
homotopy groups of weight zero of the motivic sphere spectrum
evaluated at $F$ are isomorphic to the stable homotopy groups of the
classical sphere spectrum. The role of the motivic sphere spectrum
in our context is played by the bispectrum $\Sigma_{\bb
G}^\infty\Sigma^\infty r\bb C$.

We finish the paper by proving that the triangulated category
$DK(\Re)$ of $K$-motives is fully faithfully embedded into the
homotopy category of $(S^1,\bb G)$-bispectra. We also construct an
explicit fibrant $(S^1,\bb G)$-bispectrum $\bb{KG}^{st}(A,-)$
consisting of fibrant $S^1$-spectra $\bb K^{st}(\sigma^nA,-)$,
$n\geq 0$. For this we prove the ``Cancellation Theorem" for the
spectra $\bb K^{st}(\sigma^nA,-)$ (see Theorem~\ref{novgor}). It is
reminiscent of the Cancellation Theorem proved by
Voevodsky~\cite{Voe3} for motivic cohomology. The same theorem was
proved for $K$-theory of algebraic varieties in~\cite{GP2}.

We show in Theorem~\ref{maradona} that $\bb{KG}^{st}(A,-)$ is
$(2,1)$-periodic and represents stable algebraic Kasparov $K$-theory
(cf.~\cite[6.8-6.9]{VoeICM}). Precisely, for any algebras $A,B\in\Re$ and any integers $p,q$ there is an
isomorphism
   $$\pi_{p,q}(\bb{KG}^{st}(A,B))\cong\bb K^{st}_{p-2q}(A,B).$$
As a consequence, one has that for any algebra
$B\in\Re$ and any integers $p,q$ there is an isomorphism
   $$\pi_{p,q}(\bb{KG}^{st}(k,B))\cong KH_{p-2q}(B).$$
Thus the bispectrum $\bb{KG}^{st}(k,B)$ yields an explicit model for
homotopy $K$-theory.

We should stress that the term ``motivic" is used in the paper only
for the reason that the $k[t]$-homotopy theory of algebras shares
many properties with Morel--Voevodsky's motivic homotopy theory of
smooth schemes~\cite{MV} (see remarks on page~\pageref{remmot} as
well). If there is a likelihood of confusion with other motivic
theories of commutative or non-commutative objects, the reader can
just omit the term ``motivic" everywhere.

In general, we shall not be very explicit about set-theoretical
foundations, and we shall tacitly assume we are working in some
fixed universe $\bb U$ of sets. Members of $\bb U$ are then called
{\it small sets\/}, whereas a collection of members of $\bb U$ which
does not itself belong to $\bb U$ will be referred to as a {\it
large set\/} or a {\it proper class}. If there is no likelihood of confusion, we
replace $\otimes_k$ by $\otimes$.

\section{Preliminaries}\label{prel}

In this section we collect basic facts about admissible categories
of algebras and triangulated categories associated with them. We
mostly follow~\cite{Gar,Gar1}.

\subsection{Algebraic homotopy}
Following Gersten~\cite{G} a category $\Re$ of $k$-algebras without
unit is {\it admissible\/} if it is a full subcategory of $\ahaw$
and

\begin{enumerate}

\item if $R$ is in $\Re$ and $I$ is a (two-sided) ideal of $R$ then $I$ and
$R/I$ are in $\Re$;

\item if $R$ is in $\Re$, then so is $R[x]$, the polynomial algebra in
one variable;

\item given a cartesian square
   $$\xymatrix{D\ar[r]^\rho\ar[d]_\sigma &A\ar[d]^f\\
               B\ar[r]^g &C}$$
in $\ahaw$ with $A,B,C$ in $\Re$, then $D$ is in $\Re$.

\end{enumerate}

One may abbreviate 1, 2, and 3 by saying that $\Re$ is closed under
operations of taking ideals, homomorphic images, polynomial
extensions in a finite number of variables, and fibre products. For
instance, the {\it category of commutative $k$-algebras\/} $\CAlg_k$ is
admissible.

Observe that every $k$-module $M$ can be regarded as a non-unital
$k$-algebra with trivial multiplication: $m_1\cdot m_2=0$ for all
$m_1,m_2\in M$. Then $\Mod k$ is an admissible category of
commutative $k$-algebras.

If $R$ is an algebra then the polynomial algebra $R[x]$ admits two
homomorphisms onto $R$
   $$\xymatrix{R[x]\ar@<2.5pt>[r]^{\partial_x^0}\ar@<-2.5pt>[r]_{\partial_x^1}&R}$$
where
   $$\partial_x^i|_R=1_R,\ \ \ \partial_x^i(x)=i,\ \ \ i=0,1.$$
Of course, $\partial_x^1(x)=1$ has to be understood in the sense
that $\Sigma r_nx^n\mapsto\Sigma r_n$.

\begin{defs}{\rm
Two homomorphisms $f_0,f_1:S\to R$ are {\it elementary homo\-topic},
written $f_0\sim f_1$, if there exists a homomorphism
   $$f:S\to R[x]$$
such that $\partial^0_xf=f_0$ and $\partial^1_xf=f_1$. A map $f:S\to
R$ is called an {\it elementary homotopy equivalence\/} if there is
a map $g:R\to S$ such that $fg$ and $gf$ are elementary homotopic to
$\id_R$ and $\id_S$ respectively.

}\end{defs}

For example, let $A$ be a $\bb Z_{n\geq 0}$-graded algebra, then the
inclusion $A_0\to A$ is an elementary homotopy equivalence. The
homotopy inverse is given by the projection $A\to A_0$. Indeed, the
map $A\to A[x]$ sending a homogeneous element $a_n\in A_n$ to
$a_nx^n$ is a homotopy between the composite $A\to A_0\to A$ and the
identity $\id_A$.

The relation ``elementary homotopic'' is reflexive and
symmetric~\cite[p.~62]{G}. One may take the transitive closure of
this relation to get an equivalence relation (denoted by the symbol
``$\simeq$''). Following notation of Gersten~\cite{G1}, the set of
equivalence classes of morphisms $R\to S$ is written $[R,S]$.

\begin{lem}[Gersten \cite{G1}]
Given morphisms in $\ahaw$
   $$\xymatrix{R\ar[r]^f &S\ar@/^/[r]^g \ar@/_/[r]_{g'} &T\ar[r]^h &U}$$
such that $g\simeq g'$, then $gf\simeq g'f$ and $hg\simeq hg'$.
\end{lem}

Thus homotopy behaves well with respect to composition and we have
category $Hotalg$, the {\it homotopy category of $k$-algebras},
whose objects are $k$-algebras and such that $Hotalg(R,S)=[R,S]$.
The homotopy category of an admissible category of algebras $\Re$
will be denoted by $\cc H(\Re)$. Call a homomorphism $s:A\to B$ an
{\it $I$-weak equivalence\/} if its image in $\cc H(\Re)$ is an
isomorphism. Observe that $I$-weak equivalences are those
homomorphisms which have homotopy inverses.

A diagram
   $A\bl{f}\to B\bl{g}\to C$
in $\ahaw$ is a short exact sequence if $f$ is injective, $g$ is
surjective, and the image of $f$ is equal to the kernel of $g$.

\begin{defs}{\rm
An algebra $R$ is {\it contractible\/} if $0\sim 1$; that is, if
there is a homomorphism $f:R\to R[x]$ such that $\partial^0_xf=0$
and $\partial^1_xf=1_R$.

}\end{defs}

For example, every square zero algebra $A\in\ahaw$ is contractible
by means of the homotopy $A\to A[x]$, $a\in A\mapsto ax\in A[x]$. In
other words, every $k$-module, regarded as a $k$-algebra with trivial
multiplication, is contractible.

Following Karoubi and Villamayor~\cite{KV} we define $ER$, the {\it
path algebra\/} on $R$, as the kernel of $\partial_x^0:R[x]\to R$,
so $ER\to R[x]\bl{\partial_x^0}\to R$ is a short exact sequence in
$\ahaw$. Also $\partial_x^1:R[x]\to R$ induces a surjection
   $\partial_x^1:ER\to R$
and we define the {\it loop algebra\/} $\Omega R$ of $R$ to be its
kernel, so we have a short exact sequence in $\ahaw$
   $$\Omega R\to ER\bl{\partial_x^1}\to R.$$
We call it the {\it loop extension} of $R$. Clearly, $\Omega R$ is
the intersection of the kernels of $\partial_x^0$ and
$\partial_x^1$. By~\cite[3.3]{G} $ER$ is contractible for any
algebra $R$.

\subsection{Categories of fibrant objects}

\begin{defs}{\rm
Let $\cc A$ be a category with finite products and a final object
$e$. Assume that $\cc A$ has two distinguished classes of maps,
called {\it weak equivalences\/} and {\it fibrations}. A map is
called a {\it trivial fibration\/} if it is both a weak equivalence
and a fibration. We define a {\it path space\/} for an object $B$ to
be an object $B^I$ together with maps
   $$B\lra{s}B^I\xrightarrow{(d_0,d_1)}B\times B,$$
where $s$ is a weak equivalence, $(d_0,d_1)$ is a fibration, and the
composite is the diagonal map.

Following Brown~\cite{B}, we call $\cc A$ a {\it category of fibrant
objects\/} or a {\em Brown category\/} if the following axioms are
satisfied.

(A) Let $f$ and $g$ be maps such that $gf$ is defined. If two of
$f$, $g$, $gf$ are weak equivalences then so is the third. Any
isomorphism is a weak equivalence.

(B) The composite of two fibrations is a fibration. Any isomorphism
is a fibration.

(C) Given a diagram
   $$A\bl u\longrightarrow C\bl v\longleftarrow B,$$
with $v$ a fibration (respectively a trivial fibration), the
pullback $A\times_CB$ exists and the map $A\times_CB\to A$ is a
fibration (respectively a trivial fibration).

(D) For any object $B$ in $\cc A$ there exists at least one path
space $B^I$ (not necessarily functorial in $B$).

(E) For any object $B$ the map $B\to e$ is a fibration.

}\end{defs}

\subsection{The triangulated category $D(\Re,\ff F)$}

In what follows we denote by $\ff F$ the class of $k$-split
surjective algebra homomorphisms. We shall also refer to $\ff F$ as
{\it fibrations}.

Let $\ff W$ be a class of weak equivalences in an admissible
category of algebras $\Re$ containing homomorphisms $A\to A[t]$,
$A\in\Re$, such that the triple $(\Re,\ff F,\ff W)$ is a Brown
category.

\begin{defs}{\rm
The {\it left derived category\/} $D^-(\Re,\ff F,\ff W)$ of $\Re$
with respect to $(\ff F,\ff W)$ is the category obtained from $\Re$
by inverting the weak equivalences. }\end{defs}

By~\cite{Gar1} the family of weak equivalences in the category $\cc
H\Re$ admits a calculus of right fractions. The left derived
category $D^-(\Re,\ff F,\ff W)$ (possibly ``large") is obtained from
$\cc H\Re$ by inverting the weak equivalences. The left derived
category $D^-(\Re,\ff F,\ff W)$ is left triangulated
(see~\cite{Gar,Gar1} for details) with $\Omega$ a loop functor on
it.

There is a general method of stabilizing $\Omega$ (see
Heller~\cite{Hel}) and producing a triangulated (possibly ``large")
category $D(\Re,\ff F,\ff W)$ from the left triangulated structure
on $D^-(\Re,\ff F,\ff W)$.

An object of $D(\Re,\ff F,\ff W)$ is a pair $(A,m)$ with $A\in
D^-(\Re,\ff F,\ff W)$ and $m\in\bb Z$. If $m,n\in\bb Z$ then we
consider the directed set $I_{m,n}=\{k\in\bb Z\mid m,n\leq k\}$. The
morphisms between $(A,m)$ and $(B,n)\in D(\Re,\ff F,\ff W)$ are
defined by
   $$D(\Re,\ff F,\ff W)[(A,m),(B,n)]:=\lp_{k\in I_{m,n}}D^-(\Re,\ff F,\ff W)(\Omega^{k-m}(A),\Omega^{k-n}(B)).$$
Morphisms of $D(\Re,\ff F,\ff W)$ are composed in the obvious
fashion. We define the {\it loop\/} automorphism on $D(\Re,\ff F,\ff
W)$ by $\Omega(A,m):=(A,m-1)$. There is a natural functor
$S:D^-(\Re,\ff F,\ff W)\to D(\Re,\ff F,\ff W)$ defined by
$A\longmapsto(A,0)$.

$D(\Re,\ff F,\ff W)$ is an additive category \cite{Gar,Gar1}. We
define a triangulation $\cc Tr(\Re,\ff F,\ff W)$ of the pair
$(D(\Re,\ff F,\ff W),\Omega)$ as follows. A sequence
   $$\Omega(A,l)\to (C,n)\to(B,m)\to(A,l)$$
belongs to $\cc Tr(\Re,\ff F,\ff W)$ if there is an even integer $k$
and a left triangle of representatives
$\Omega(\Omega^{k-l}(A))\to\Omega^{k-n}(C)\to\Omega^{k-m}(B)\to\Omega^{k-l}(A)$
in $D^-(\Re,\ff F,\ff W)$. Then the functor $S$ takes left triangles
in $D^-(\Re,\ff F,\ff W)$ to triangles in $D(\Re,\ff F,\ff W)$.
By~\cite{Gar,Gar1} $\cc Tr(\Re,\ff F,\ff W)$ is a triangulation of
$D(\Re,\ff F,\ff W)$ in the classical sense of Verdier~\cite{Ver}.

By an {\it $\ff F$-extension\/} or just {\it extension\/} in $\Re$
we mean a short exact sequence of algebras
   \begin{equation*}
    (E):A\to B\bl\alpha\to C
   \end{equation*}
such that $\alpha\in\ff F$. Let $\cc E$ be the class of all $\ff
F$-extensions in $\Re$.

\begin{defs}{\rm Following Corti\~nas--Thom~\cite{CT} a {\it ($\ff
F$-)excisive homology theory\/} on $\Re$ with values in a
triangulated category $(\cc T,\Omega)$ consists of a functor
$X:\Re\to \cc T$, together with a collection $\{\partial_E:E\in\cc
E\}$ of maps $\partial_E^X=\partial_E\in{\cc T}(\Omega X(C), X(A))$.
The maps $\partial_E$ are to satisfy the following requirements. \sn
\noindent{(1)} For all $E\in \cc E$ as above,
\[
\xymatrix{\Omega
X(C)\ar[r]^{\partial_E}&X(A)\ar[r]^{X(f)}&X(B)\ar[r]^{X(g)}& X(C)}
\]
is a distinguished triangle in $\cc T$. \sn \noindent{(2)} If
\[
\xymatrix{(E): &A\ar[r]^f\ar[d]_\alpha& B\ar[r]^g\ar[d]_\beta& C\ar[d]_\gamma\\
          (E'):&A'\ar[r]^{f'}& B'\ar[r]^{g'}& C'}
\]
is a map of extensions, then the following diagram commutes
\[
\xymatrix{{\Omega} X(C)\ar[d]_{{\Omega} X(\gamma)}\ar[r]^{\partial_E}& X(A)\ar[d]^{X(\alpha)}\\
 \Omega X(C')\ar[r]_{\partial_{E'}}& X(A).}
\]
We say that the functor $X:\Re\to\cc T$ is {\it homotopy
invariant\/} if it maps homotopic homomomorphisms to equal maps, or
equivalently, if for every $A\in\ahaw$, $X$ maps the inclusion
$A\subset A[t]$ to an isomorphism.

}\end{defs}

Denote by $\ff W_{\triangle}$ the class of homomorphisms $f$ such
that $X(f)$ is an isomorphism for any excisive, homotopy invariant
homology theory $X:\Re\to\cc T$. We shall refer to the maps from
$\ff W_{\triangle}$ as {\it stable weak equivalences}. The triple
$(\Re,\ff F,\ff W_{\triangle})$ is a Brown category. In what follows
we shall write $D^-(\Re,\ff F)$ and $D(\Re,\ff F)$ to denote
$D^-(\Re,\ff F,\ff W_{\triangle})$ and $D(\Re,\ff F,\ff
W_{\triangle})$ respectively, dropping $\ff W_{\triangle}$ from the
notation.

By~\cite{Gar1} the canonical functor
   $$\Re\to D(\Re,\ff F)$$
is the universal excisive, homotopy invariant homology theory on
$\Re$.

\section{Homotopy theory of algebras}\label{homotopy}

Let $\Re$ be a {\it small\/} admissible category of algebras. We
shall work with various model category structures for the category
of simplicial functors on $\Re$. We mostly adhere
to~\cite{Gar,Gark}.

\subsection{The categories of pointed simplicial functors $U_\bullet\Re$}

Throughout this paper we work with a model category $U_\bullet\Re$.
To define it, we first enrich $\Re$ over pointed simplicial sets
$\bb S_\bullet$. Given an algebra $A\in\Re$, denote by $rA$ the
representable functor $\Hom_{\Re}(A,-)$. Let $\Re_\bullet$ have the
same objects as $\Re$ and have pointed simplicial sets of morphisms
being the $rA(B)=\Hom_{\Re}(A,B)$ pointed at zero. Denote by
$U_\bullet\Re$ the category of $\bb S_\bullet$-enriched functors
from $\Re_\bullet$ to $\bb S_\bullet$. One easily checks that
$U_\bullet\Re$ can be regarded as the category of covariant pointed
simplicial functors $X:\Re\to\bb S_\bullet$ such that $X(0)=*$.

By~\cite[4.2]{DRO} we define the projective model structure on
$U_\bullet\Re$. This is a proper, simplicial, cellular model
category with weak equivalences and fibrations being defined
objectwise, and cofibrations being those maps having the left
lifting property with respect to trivial fibrations.

The class of projective cofibrations for $U_\bullet\Re$ is generated
by the set
$$
I_{U_\bullet\Re}=\{rA\wedge(\partial\Delta^{n}\subset\Delta^{n})_+\}^{n\geq
0}$$ indexed by $A\in\Re$. Likewise, the class of acyclic projective
cofibrations is generated by
$$
J_{U_\bullet\Re}=\{rA\wedge(\Lambda^{k}_n\subset\Delta^{n})_+\}^{n>0}_{0\leq
k\leq n}.
$$

Given $\cc X,\cc Y\in U_\bullet\Re$ the pointed function complex
$\Map_\bullet(\cc X,\cc Y)$ is defined as
   $$\Map_\bullet(\cc X,\cc Y)_n=\Hom_{U_\bullet\Re}(\cc X\wedge\Delta^n_+,\cc Y),\quad n\geq 0.$$
By~\cite[2.1]{DRO} there is a natural isomorphism of pointed
simplicial sets
   $$\Map_\bullet(rA,\cc X)\cong\cc X(A)$$
for all $A\in\Re$ and $\cc X\in U_\bullet\Re$.

Recall that the model category $U\Re$ of functors from $\Re$ to
unpointed simplicial sets $\bb S$ is defined in a similar fashion
(see~\cite{Gar}). Since we mostly work with spectra in this paper,
the category of spectra associated with $U_\bullet\Re$ is
technically more convenient than the category of spectra associated
with $U\Re$.

\subsection{The model categories $U_\bullet\Re_I,U_\bullet\Re_J,U_\bullet\Re_{I,J}$}

Let $I=\{i=i_A:r(A[t])\to r(A)\mid A\in\Re\}$, where each $i_A$ is
induced by the natural homomorphism $i:A\to A[t]$. Recall that a
functor $F:\Re\to\bb S_\bullet/Spectra$ is {\it homotopy
invariant\/} if $F(A)\to F(A[t])$ is a weak equivalence for all
$A\in\Re$. Consider the projective model structure on
$U_\bullet\Re$. We shall refer to the $I$-local equivalences as
(projective) $I$-weak equivalences. Denote by $U_\bullet\Re_I$ the
model category obtained from $U_\bullet\Re$ by Bousfield
localization with respect to the family $I$. Notice that any
objectwise fibrant homotopy invariant functor $F\in U_\bullet\Re$ is
an $I$-local object, hence fibrant in $U_\bullet\Re_I$.

Let us introduce the class of excisive functors on $\Re$. They look
like flasque presheaves on a site defined by a cd-structure in the
sense of Voevodsky~\cite[section 3]{V}.

\begin{defs}{\rm
A simplicial functor $\cc X\in U_\bullet\Re$ is called {\it
excisive\/} with respect to $\ff F$ if for any cartesian square in
$\Re$
   $$\xymatrix{
      D\ar[r]\ar[d]&A\ar[d]\\
      B\ar[r]^f&C
     }$$
with $f$ a fibration (call such squares {\it distinguished}) the
square of simplicial sets
   $$\xymatrix{
      \cc X(D)\ar[r]\ar[d]&{\cc X(A)}\ar[d]\\
      {\cc X(B)}\ar[r]&\cc X(C)
     }$$
is a homotopy pullback square. It immediately follows from the
definition that every excisive object takes $\ff F$-extensions in
$\Re$ to homotopy fibre sequences of simplicial sets.

}\end{defs}

Let $\alpha$ denote a distinguished square in $\Re$ as shown.
   $$\xymatrix{
      D\ar[r]\ar[d]&A\ar[d]\\
      B\ar[r]&C}$$
Let us apply the simplicial mapping cylinder construction $\cyl$ to
$\alpha$ and form the pushouts:
   $$\xymatrix{rC\ar[r]\ar[d]&\cyl(rC\to rA)\ar[d]\ar[r]&rA\ar[d]\\
               rB\ar[r]&\cyl(rC\to rA)\sqcup_{rC}rB\ar[r]&rD}$$
Note that $rC\to\cyl(rC\rightarrow rA)$ is a projective cofibration
between (projective) cofibrant objects of $U_\bullet\Re$. Thus
$s(\alpha)=\cyl(rC\to rA)\sqcup_{rC}rB$ is (projective)
cofibrant \cite[1.11.1]{Hov}. For the same reasons, applying the
simplicial mapping cylinder to $s(\alpha)\rightarrow rD$ and setting
$t(\alpha)=\cyl\bigl(s(\alpha)\to rD\bigr)$ we get a projective
cofibration
   $$\xymatrix{\cyl(\alpha)\colon s(\alpha)\ar[r] & t(\alpha).}$$

Let $J^{\cyl(\alpha)}_{U_\bullet\Re}$ consists of all pushout
product maps
$$
\xymatrix{
s(\alpha)\wedge\Delta^n_+\sqcup_{s(\alpha)\wedge\partial\Delta^n_+}
t(\alpha)\wedge\partial\Delta^n_+\ar[r] & t(\alpha)\wedge\Delta^n_+}
$$
and let $J= J_{U_\bullet\Re}\cup J^{\cyl(\alpha)}_{U_\bullet\Re}$.
Denote by $U_\bullet\Re_J$ the model category obtained from
$U_\bullet\Re$ by Bousfield localization with respect to the family
$J$. It is directly verified that $\cc X\in U_\bullet\Re$ is
$J$-local if and only if it has the right lifting property with
respect to $J$. Also, $\cc X$ is $J$-local if and only if it is
objectwise fibrant and excisive~\cite[4.3]{Gar}.

Finally, let us introduce the model category $U_\bullet\Re_{I,J}$.
It is, by definition, the Bousfield localization of $U_\bullet\Re$
with respect to $I\cup J$. The weak equivalences (trivial
cofibrations) of $U_\bullet\Re_{I,J}$ will be referred to as
(projective) $(I,J)$-weak equivalences ((projective) $(I,J)$-trivial
cofibrations). By~\cite[4.5]{Gar} a functor $\cc X\in U_\bullet\Re$
is $(I,J)$-local if and only if it is objectwise fibrant, homotopy
invariant and excisive.

\begin{rem}\label{remmot}{\rm
The model category $U_\bullet\Re_{I,J}$ can also be regarded as a
kind of unstable motivic model category associated with $\Re$.
Indeed, the construction of $U_\bullet\Re_{I,J}$ is similar to
Morel--Voevodsky's unstable motivic theory for smooth schemes $Sm/F$
over a field $F$~\cite{MV}. If we replace $I$ by
   $$I'=\{X\times\bb A^1\lra{pr}X\mid X\in Sm/F\}$$
and the family of distinguished squares by the family of elementary
Nisnevich squares and get the corresponding family $J'$ associated
to it, then one of the equivalent models for Morel--Voevodsky's
unstable motivic theory is obtained by Bousfield localization of
simplicial pre\-sheaves with respect to $I'\cup J'$.

For this reason, $U_\bullet\Re_{I,J}$ can also be called the
category of (pointed) {\it motivic\/} spaces, where each algebra $A$
is identified with the pointed motivic space $rA$. One can also
refer to $(I,J)$-weak equivalences as {\it motivic weak
equivalences}.

}\end{rem}

\subsection{Monoidal structure on $U_\bullet\Re$}

In this section we mostly follow~\cite[section~2.1]{O}. Suppose
$\Re$ is tensor closed, that is $k\in \Re$ and $A\otimes B\in\Re$
for all $A,B\in\Re$. We introduce the monoidal product $\cc
X\otimes\cc Y$ of $\cc X$ and $\cc Y$ in $U_\bullet\Re$ by the
formulas
\begin{equation*}
\cc X\otimes\cc Y(A)= \underset{A_1\otimes A_2\rightarrow
A}{\colim} \cc X(A_1)\wedge\cc Y(A_2).
\end{equation*}
The colimit is indexed on the category with objects $\alpha\colon
A_1\otimes A_2\rightarrow A$ and maps the pairs of maps
$(\phi,\psi)\colon (A_1,A_2)\rightarrow (A'_1,A'_2)$ such that
$\alpha'(\psi\otimes\phi)=\alpha$. By functoriality of colimits it
follows that $\cc X\otimes\cc Y$ is in $U_\bullet\Re$.

The tensor product can also be defined by the formula
\begin{equation*}
\cc X\otimes\cc Y(A)= \int^{A_1,A_2\in\Re} \bigl(\cc X(A_1)\wedge\cc
Y(A_2)\bigr) \wedge\Hom_{\Re}(A_1\otimes A_2,A).
\end{equation*}
This formula is obtained from a theorem of Day~\cite{Day}, which
also asserts that the triple $(U_\bullet\Re,\otimes,r(k))$ forms a
closed symmetric monoidal category.

The internal Hom functor, right adjoint to $\cc X\otimes-$, is given
by
\begin{equation*}
\underline{\Hom}(\cc X,\cc Y)(A)= \int_{B\in\Re}\Map_\bullet(\cc
X(B),\cc Y(A\otimes B)),
\end{equation*}
where $\Map_\bullet$ stands for the function complex in $\bb
S_\bullet$.

So there exist natural isomorphisms
\begin{equation*}
\underline{\Hom}(\cc X\otimes\cc Y,\cc
Z)\cong\underline{\Hom}\bigl(\cc X,\underline{\Hom}(\cc Y,\cc
Z)\bigr)
\end{equation*}
and
\begin{equation*}
\underline{\Hom}(r(k),\cc Z)\cong\cc Z.
\end{equation*}

Concerning smash products of representable functors, one has a
natural isomorphism
\begin{equation*}
rA\otimes rB\cong r(A\otimes B),\quad A,B\in\Re.
\end{equation*}
Note as well that, for pointed simplicial sets $K$ and $L$, one has
$K\otimes L=K\wedge L$.

We recall a pointed simplicial set tensor and cotensor structure on
$U_\bullet\Re$. If $\cc X$ and $\cc Y$ are in $U_\bullet\Re$ and $K$
is a pointed simplicial set, the tensor $\cc X\otimes K$ is given by
\begin{equation*}
\cc X\otimes K(A)= \cc X(A)\wedge K
\end{equation*}
and the cotensor $\cc Y^K$ is given in terms of the ordinary
function complex:
\begin{equation*}
\cc Y^K(A)={\Map_\bullet}\bigl(K,\cc Y(A)\bigr).
\end{equation*}

The function complex $\Map_\bullet(\cc X,\cc Y)$ of $\cc X$ and $\cc
Y$ is defined by setting
\begin{equation*}
\Map_\bullet(\cc X,\cc Y)_n=\Hom_{U_\bullet\Re}(\cc
X\otimes\Delta^n_+,\cc Y).
\end{equation*}
By the Yoneda lemma there exists a natural isomorphism of pointed
simplicial sets
\begin{equation*}
\Map_\bullet(rA,\cc Y)\cong\cc Y(A).
\end{equation*}
Using these definitions $U_\bullet\Re$ is enriched in pointed
simplicial sets $\bb S_\bullet$. Moreover, there are natural
isomorphisms of pointed simplicial sets
\begin{equation*}
\Map_\bullet(\cc X\otimes K,\cc
Y)\cong\Map_\bullet\bigl(K,\Map_\bullet(\cc X,\cc
Y)\bigr)\cong\Map_\bullet(\cc X,\cc Y^K).
\end{equation*}
It is also useful to note that
\begin{equation*}
\underline{\Hom}(\cc X,\cc Y)(A)=\Map_\bullet\bigl(\cc X,\cc
Y(A\otimes-)\bigr).
\end{equation*}
and
\begin{equation*}
\underline{\Hom}(rB,\cc Y)=\cc Y(-\otimes B).
\end{equation*}

It can be shown similarly to~\cite[3.10; 3.43; 3.89]{O} that the
model categories $U_\bullet\Re$, $U_\bullet\Re_I$, $U_\bullet\Re_J$,
$U_\bullet\Re_{I,J}$ are monoidal.

\section{Unstable algebraic Kasparov $K$-theory}

Let $\cc U$ be an arbitrary category and let $\Re$ be an admissible
category of $k$-algebras. Suppose that there are functors
$F:\Re\to\cc U$ and $\wt T:\cc U\to\Re$ such that $\wt T$ is left
adjoint to $F$. We denote $\wt TFA$, for $ A\in\Re$, by $TA$ and the
counit map $\wt TFA\to A$ by $\eta_A$. If $X\in\Ob\cc U$ then the
unit map $X\to F\wt TX$ is denoted by $i_X$. We note that the
composition
   $$FA\xrightarrow{i_{FA}}F\wt TFA\xrightarrow{F\eta_A}FA$$
equals $1_{FA}$ for every $A\in\Re$, and hence $F\eta_A$ splits in
$\cc U$. We call an admissible category of $k$-algebras {\it
$T$-closed\/} if $TA\in\Re$ for all $A\in\Re$.

\begin{lem}\label{ta}
Suppose $\cc U$ is either a full subcategory of the category of sets
or a full subcategory of the category of $k$-modules. Suppose as
well that $F:\Re\to\cc U$ is the forgetful functor. Then for every
$A\in\Re$ the algebra $TA$ is contractible, i.e. there is a
contraction $\tau:TA\to TA[x]$ such that
$\partial_x^0\tau=0,\partial_x^1\tau=1$. Moreover, the contraction
is functorial in $A$.
\end{lem}

\begin{proof}
Consider a map $u:FTA\to FTA[x]$ sending an element $b\in FTA$ to
$bx\in FTA[x]$. By assumption, $u$ is a morphisms of $\cc U$. The
desired contraction $\tau$ is uniquely determined by the map $u\circ
i_{FA}:FA\to FTA[x]$. By using elementary properties of adjoint
functors, one can show that $\partial_x^0\tau=0$ and
$\partial_x^1\tau=1$.
\end{proof}

Throughout this paper, whenever we deal with a $T$-closed admissible
category of $k$-algebras $\Re$ we assume to be fixed an underlying
category $\cc U$, which is a full subcategory of $\Mod k$.

\begin{exs}{\rm
(1) Let $\Re=\ahaw$. Given an algebra $A$, consider the algebraic
tensor algebra
   $$TA=A\oplus A\otimes A\oplus A^{\otimes^3}\oplus\cdots$$
with the usual product given by concatenation of tensors. In Cuntz's
treatment of bivariant $K$-theory~\cite{Cu2,Cu,Cu1}, tensor algebras
play a prominent role.

There is a canonical $k$-linear map $A\to TA$ mapping $A$ into the
first direct summand. Every $k$-linear map $s:A\to B$ into an
algebra $B$ induces a homomorphism $\gamma_s:TA\to B$ defined by
   $$\gamma_s(x_1\otimes\cdots\otimes x_n)=s(x_1)s(x_2)\cdots s(x_n).$$
Plainly $\Re$ is $T$-closed.

(2) If $\Re=\cahaw$ then
   $$T(A)=Sym(A)=\oplus_{n\geq 1}S^nA,\quad S^nA=A^{\otimes n}/\langle a_1\otimes\cdots\otimes a_n
     -a_{\sigma(1)}\otimes\cdots\otimes a_{\sigma(n)}\rangle,\quad\sigma\in\Sigma_n,$$
the symmetric algebra of $A$, and $\Re$ is $T$-closed.

}\end{exs}

We have a natural extension of algebras
   $$0\lra{}JA\lra{\iota_A}TA\lra{\eta_A}A\lra{}0.$$
Here $JA$ is defined as $\kr\eta_A$\label{jjjj}. Clearly, $JA$ is
functorial in $A$.

Given a small $T$-closed admissible category of $k$-algebras $\Re$,
we denote by $Sp(\Re)$ the category of $S^1$-spectra in the sense of
Hovey~\cite{H} associated with the model category
$U_\bullet\Re_{I,J}$. Recall that a spectrum consists of sequences
$\cc E=(\cc E_{n})_{n\geq 0}$ of pointed simplicial functors in
$U_\bullet\Re$ equipped with structure maps $\sigma_{n}^{\cc
E}:\Sigma\cc E_{n}\rightarrow\cc E_{n+1}$ where $\Sigma=-\wedge S^1$
is the suspension functor. A map $f:\cc E\rightarrow\cc F$ of
spectra consists of compatible maps $f_{n}:\cc E_{n}\to\cc F_{n}$ in
the sense that the diagrams
$$
\xymatrix{
\Sigma\cc E_{n}\ar[d]_-{\Sigma f_{n}}\ar[r]^-{\sigma_{n}^{\cc E}} & \cc E_{n+1}\ar[d]^-{f_{n+1}} \\
\Sigma\cc F_{n}\ar[r]^-{\sigma_{n}^{\cc F}} & \cc F_{n+1} }
$$
commute for all $n\geq 0$. The category $Sp(\Re)$ is endowed with
the stable model structure (see~\cite{H} for details).

Given an algebra $A\in\Re$, we denote by $\Sigma^\infty rA$ the
suspension spectrum associated with the functor $rA$ pointed at
zero. By definition, $(\Sigma^\infty rA)_n=rA\wedge S^n$ with
obvious structure maps.

In order to define one of the main spectra of the paper $\cc R(A)$
associated to an algebra $A\in\Re$, we have to recall some
definitions from~\cite{Gark}.

For any $B\in\Re$ we define a simplicial algebra
   $$B^{\Delta}:[n]\mapsto B^{{\Delta}^n}:=B[t_0,\dots, t_n]/\langle 1-\sum_i t_i\rangle\ \ \ (\cong
     B[t_1,\ldots,t_n]).$$
Given a map of posets $\alpha:[m]\to[n]$, the map
$\alpha^*:B^{\Delta^n}\to B^{\Delta^m}$ is defined by
$\alpha^*(t_j)=\sum_{\alpha(i)=j}t_i$.
We have that $B^\Delta\cong B\otimes k^\Delta$ and $B^\Delta$ is
pointed at zero.

For any pointed simplicial set $X\in\bb S_\bullet$, we denote by
$B^\Delta(X)$ the simplicial algebra $\Map_\bullet(X,B^\Delta)$. The
simplicial algebra associated to any unpointed simplicial set and
any simplicial algebra is defined in a similar way. By $\bb
B^\Delta(X)$ we shall mean the pointed simplicial ind-algebra
   $$B^\Delta(X)\to B^\Delta(\sd^1X)\to B^\Delta(\sd^2X)\to\cdots$$
In particular, one defines the ``path space" simplicial ind-algebra
$P\bb B^\Delta$. We shall also write $\bb B^\Delta(\Omega^n)$ to
denote $\bb B^\Delta(S^n)$, where $S^n=S^1\wedge\cdots\wedge S^1$ is
the simplicial $n$-sphere. For any $A\in\Re$ we denote by
$\Hom_{\inda}(A,\bb B^\Delta(\Omega^n))$ the colimit of the sequence
in $\bb S_\bullet$
   $$\Hom_{\ahaw}(A,B^\Delta(S^n))\to\Hom_{\ahaw}(A,B^\Delta(\sd^1S^n))\to\Hom_{\ahaw}(A,B^\Delta(\sd^2S^n))\to\cdots$$

The natural simplicial map $d_1:P\bb B^\Delta(\Omega^n)\to\bb
B^\Delta(\Omega^n)$ has a natural $k$-linear splitting described
below. Let $\mathbf{t}\in
P\Bbbk^\Delta(\Delta^1\times\bl{n}\cdots\times\Delta^1)_0$\label{elt}
stand for the composite map
   $$\sd^m(\Delta^1\times\bl{n+1}\cdots\times\Delta^1)\lra{pr}\sd^m\Delta^1\to\Delta^1\bl t\to k^\Delta,$$
where $pr$ is the projection onto the $(n+1)$th direct factor
$\Delta^1$ and $t=t_0\in k^{\Delta^1}$. The element $\mathbf{t}$ can
be regarded as a 1-simplex of the unital ind-algebra
$\Bbbk^\Delta(\Delta^1\times\bl{n}\cdots\times\Delta^1)$ such that
$\partial_0(\mathbf{t})=0$ and $\partial_1(\mathbf{t})=1$. Let
$\imath:\bb B^\Delta(\Omega^n)\to(\bb
B^\Delta(\Omega^n))^{\Delta^1}$ be the natural inclusion.
Multiplication with $\mathbf{t}$ determines a $k$-linear map $(\bb
B^\Delta(\Omega^n))^{\Delta^1}\lra{\mathbf{t}\cdot}P\bb
B^\Delta(\Omega^n)$. Now the desired $k$-linear splitting $\bb
B^\Delta(\Omega^{n})\lra{\upsilon}P\bb B^\Delta(\Omega^{n})$ of
simplicial ind-modules is defined as
   $$\upsilon:=\mathbf{t}\cdot\imath.$$

If we consider $\bb B^\Delta(\Omega^n)$ as a $(\bb Z_{\geq
0}\times\Delta)$-diagram in $\Re$, then there is a commutative
diagram of extensions for $(\bb Z_{\geq 0}\times\Delta)$-diagrams
   $$\xymatrix{J\bb B^\Delta(\Omega^n)\ar[d]_{\xi_\upsilon}\ar[r]&T\bb B^\Delta(\Omega^n)\ar[r]\ar[d]&\bb B^\Delta(\Omega^n)\ar@{=}[d]\\
               \bb B^\Delta(\Omega^{n+1})\ar[r]&P\bb B^\Delta(\Omega^n)\ar[r]^{d_1}&\bb B^\Delta(\Omega^n),}$$
where the map $\xi_{\upsilon}$ is uniquely determined by the
$k$-linear splitting $\upsilon$. For every element
$f\in\Hom_{\inda}(J^nA,\bb B^\Delta(\Omega^n))$ one sets:
   $$\varsigma(f):=\xi_{\upsilon}\circ J(f)\in\Hom_{\inda}(J^{n+1}A,\bb B^\Delta(\Omega^{n+1})).$$

The spectrum $\cc R(A)$ is defined at every $B\in\Re$ as the
sequence of spaces pointed at zero
   $$\Hom_{\inda}(A,\bb B^\Delta),\Hom_{\inda}(JA,\bb B^\Delta),\Hom_{\inda}(J^2A,\bb B^\Delta),\ldots$$
By~\cite[section~2]{Gark} each $\cc R(A)_n(B)$ is a fibrant
simplicial set and
   $$\Omega^k\cc R(A)_0(B)=\Hom_{\inda}(A,\bb B^\Delta(\Omega^k)).$$
Each structure map $\sigma_n:\cc R(A)_n\wedge S^1\to\cc R(A)_{n+1}$
is defined at $B$ as adjoint to the map
$\varsigma:\Hom_{\inda}(J^nA,\bb
B^\Delta)\to\Hom_{\inda}(J^{n+1}A,\bb B^\Delta(\Omega))$.

For every $A\in\Re$ there is a natural map in $Sp(\Re)$
   $$i:\Sigma^\infty rA\to\cc R(A)$$
functorial in $A$.

\begin{defs}{\rm (see~\cite{Gark})
(1) Given two $k$-algebras $A,B\in\Re$, the {\it unstable algebraic
Kasparov $K$-theory space\/} $\mathcal{K}(A,B)$
is the fibrant space
   $$\lp_n\Hom_{\inda}(J^nA,\bb B^\Delta(\Omega^n)),$$
where the colimit maps are given by $\xi_\upsilon$-s and $JA$ is as
defined on page~\pageref{jjjj}. Its homotopy groups will be denoted
by $\mathcal{K}_n(A,B)$, $n\geq 0$. The simplicial functor
$\mathcal{K}(A,-)$ is fibrant in $U_\bullet(\Re)_{I,J}$
by~\cite[section~4]{Gark}. Also, there is a natural isomorphism of
simplicial sets
   $$\cc K(A,B)\cong\Omega\cc K(JA,B).$$
In particular, $\cc K(A,B)$ is an infinite loop space with $\cc
K(A,B)$ which simplicially isomorphic to $\Omega^n\cc K(J^nA,B)$
(see~\cite[5.1]{Gark}).

(2) The {\it unstable algebraic Kasparov $KK$-theory spectrum\/} of
$(A,B)$ consists of the sequence of spaces
   $$\mathcal{K}(A,B),\mathcal{K}(JA,B),\mathcal{K}(J^2A,B),\ldots$$
together with the natural isomorphisms
$\mathcal{K}(J^nA,B)\cong\Omega\mathcal{K}(J^{n+1}A,B)$. It forms an
$\Omega$-spectrum which we also denote by $\mathbb{K}(A,B)$. Its
homotopy groups will be denoted by $\mathbb{K}_n(A,B)$, $n\in\bb Z$.
Observe that $\mathbb{K}_n(A,B)\cong\mathcal{K}_n(A,B)$ for any
$n\geq 0$ and $\mathbb{K}_n(A,B)\cong\mathcal{K}_0(J^{-n}A,B)$ for
any $n<0$.

}\end{defs}

There is a natural map of spectra
   $$j:\cc R(A)\to\mathbb{K}(A,-).$$
By~\cite[section~6]{Gark} this is a stable equivalence and
$\mathbb{K}(A,-)$ is a fibrant object of $Sp(\Re)$. In fact for any
algebra $B\in\Re$ the map
   $$j:\cc R(A)(B)\to\mathbb{K}(A,B)$$
is a stable equivalence of ordinary spectra.

The following theorem is crucial in our analysis. It states that
$\bb K(A,-)$ is a fibrant replacement of $\Sigma^\infty rA$ in
$Sp(\Re)$.

\begin{thm}\label{mainres}
Given $A\in\Re$ the map $i:\Sigma^\infty rA\to\cc R(A)$ is a level
$(I,J)$-weak equivalence, and therefore the composite map
   $$\Sigma^\infty rA\bl i\to\cc R(A)\bl j\to\bb K(A,-)$$
is a stable equivalence in $Sp(\Re)$, functorial in $A$.
\end{thm}

\begin{proof}
Recall that for any functor $F$ from rings to simplicial sets,
$Sing(F)$ is defined at each ring $R$ as the diagonal of the
bisimplicial set $F(R[\Delta])$. The map $i_0:(\Sigma^\infty
rA)_0\to\cc R(A)_0$ equals $rA\to Ex^\infty\circ Sing(rA)$, which is
an $I$-weak equivalence by~\cite[3.8]{Gar}. Let us show that
$i_1:rA\wedge S^1\to\cc R(A)_1=Ex^\infty\circ Sing(r(JA))$ is an
$(I,J)$-weak equivalence. It is fully determined by the element
$\rho_A:JA\to\Omega A$, which is a zero simplex of
$\Omega(Ex^\infty\circ Sing(r(JA))(A))$, coming from the adjunction
isomorphism
   $$\Map_{\bullet}(rA\wedge S^1,Ex^\infty\circ Sing(r(JA)))\cong\Omega(Ex^\infty\circ Sing(r(JA))(A)).$$
Let $(I,0)$ denote $\Delta[1]$ pointed at 0. Consider a commutative
diagram of cofibrant objects in $U_\bullet\Re$
   $$\xymatrix{rA\ar[d]_{\eta_A^*}\ar@{ >->}[r]^(.4){\nu}&rA\wedge(I,0)\ar[d]\ar@{->>}[r]&rA\wedge S^1\ar@{=}[d]\\
               r(TA)\ar@{ >->}[r]&\cc X\ar@{->>}[r]^(.4)\alpha&rA\wedge S^1}$$
where the left square is pushout, the left map is induced by the
canonical homomorphism $\eta_A:TA\to A$ and $\nu$ is induced by the
natural inclusion $d^0:\Delta[0]\to\Delta[1]$. Lemma~\ref{ta}
implies $r(TA)$ is weakly equivalent to zero in $U_\bullet\Re_I$. It
follows that $\alpha$ is an $I$-weak equivalence.

By the universal property of pullback diagrams there is a unique
morphism $\sigma:\cc X\to r(JA)$ whose restriction to $r(TA)$ equals
$\iota^*_A$, where $\iota_A=\kr\eta_A$, which makes the diagram
   $$\xymatrix@!0{
     &rA\wedge(I,0)\ar[rrr]\ar'[d][dd] &&& \cc X\ar[dd]^\sigma\\
     rA\ar[ur]\ar[rrr]\ar[dd]_1 &&& r(TA)\ar[ur]\ar[dd]\\
     &pt\ar'[rr][rrr] &&& r(JA)\\
     rA\ar[rrr]_{\eta_A^*}\ar[ur] &&& r(TA)\ar[ur]}$$
commutative. Since the upper and the lower squares are homotopy
pushouts in $U_\bullet\Re_{I,J}$ and $rA\wedge(I,0)$ is weakly
equivalent to zero, it follows from~\cite[13.5.10]{Hir} that
$\sigma$ is an $(I,J)$-weak equivalence. Therefore the composite
map, we shall denote it by $\rho$,
   $$\cc X\bl\sigma\to r(JA)\to\cc R(A)_1$$
is an $(I,J)$-weak equivalence, where the right map is the natural
$I$-weak equivalence.

Let $\cc R(A)_1[x]\in U_\bullet\Re$ be a simplicial functor defined
as
   $$\cc R(A)_1[x](B)=\Hom_{\inda}(JA,\bb B^\Delta[x])=Ex^\infty\circ\Hom_{\ahaw}(JA,B[x]^\Delta),\quad B\in\Re.$$
There is a natural map $s:\cc R(A)_1\to\cc R(A)_1[x]$, induced by
the monomorphism $B\to B[x]$ at each $B$. It follows
from~\cite[3.2]{Gar} that this map is a weak equivalence in
$U_\bullet\Re$. The evaluation homomorphisms
$\partial_x^0,\partial_x^1:B[x]\to B$ induce a map
$(\partial_x^0,\partial_x^1):\cc R(A)_1[x]\to\cc R(A)_1\times\cc
R(A)_1$, whose composition with $s$ is the diagonal map $\cc
R(A)_1\to\cc R(A)_1\times\cc R(A)_1$. We see that $\cc R(A)_1[x]$ is
a path object for the projectively fibrant object $\cc R(A)_1$.

If we constructed a homotopy $H:\cc X\to\cc R(A)_1[x]$ such that
$\partial_x^0H=i_1\alpha$ and $\partial_x^1H=\rho$ it would follow
that $i_1\alpha$, being homotopic to the $(I,J)$-weak equivalence
$\rho$, is an $(I,J)$-weak equivalence. Since also $\alpha$ is an
$(I,J)$-weak equivalence, then so would be $i_1$.

The desired map $H$ is uniquely determined by maps $h_1:r(TA)\to\cc
R(A)_1[x]$ and $h_2:rA\wedge(I,0)\to\cc R(A)_1[x]$ such that
$h_1\eta^*_A=h_2\nu$ is defined as follows. The map $h_1$ is
uniquely determined by the homomorphism $JA\to TA[x]$ which is the
composition of $\iota_A$ and the contraction homomorphism
$\tau:TA\to TA[x]$, functorial in $A$, that exists by
Lemma~\ref{ta}. The map $h_2$ is uniquely determined by the
one-simplex $JA\to A[\Delta^1][x]$ of
$Ex^\infty\circ\Hom_{\ahaw}(JA,A[x]^\Delta)$ which is the
composition of $\rho_A:JA\to\Omega A=(t^2-t)A[t]\subset A[\Delta^1]$
and the homomorphism $\omega:A[\Delta^1]\to A[\Delta^1][x]$ sending
the variable $t$ to $1-(1-t)(1-x)$.

Thus we have shown that
   $$i_1:rA\wedge S^1\to\cc R(A)_1$$
is an $(I,J)$-weak equivalence. It follows that the composite map
   $$rA\wedge S^1\xrightarrow{i_0\wedge S^1}\cc R(A)_0\wedge S^1\lra{\sigma_0}\cc R(A)_1,$$
which is equal to $i_1$, is an $(I,J)$-weak equivalence. Hence
$\sigma_0$ is an $(I,J)$-weak equivalence, because $i_0\wedge S^1$
is an $I$-weak equivalence. More generally, one gets that every
structure map
   $$\cc R(A)_n\wedge S^1\xrightarrow{\sigma_n}\cc R(A)_{n+1}$$
is an $(I,J)$-weak equivalence.

By induction, assume that $i_n:rA\wedge S^n\to\cc R(A)_n$ is an
$(I,J)$-weak equivalence. Then $i_n\wedge S^1$ is an $(I,J)$-weak
equivalence, and hence so is $i_{n+1}=\sigma_n\circ(i_n\wedge S^1)$.
\end{proof}

Denote by $\shs(\Re)$ the stable homotopy category of $Sp(\Re)$.
Since the endofunctor $-\wedge S^1$ is an equivalence on $\shs(\Re)$
by~\cite{H}, it follows from~\cite[Ch.~7]{Hov} that $\shs(\Re)$ is a
triangulated category. Moreover, it is compactly generated with
compact generators $\{(\Sigma^\infty rA)[n]\}_{A\in\Re,n\in\bb Z}$.

\begin{cor}\label{rbra}
$\{\Sigma^\infty rA[n]\}_{A\in\Re,n\in\bb Z}$ is a family of compact
generators for $\shs(\Re)$. Moreover, there is a natural isomorphism
   $$\shs(\Re)(\Sigma^\infty rB[n],\Sigma^\infty rA)\cong\bb K_n(A,B)$$
for all $A,B\in\Re$ and $n\in\bb Z$.
\end{cor}

Denote by $\cc S$ the full subcategory of $\shs(\Re)$ whose objects
are $\{\Sigma^\infty rA[n]\}_{A\in\Re,n\in\bb Z}$. The next
statement gives another description of the triangulated category
$D(\Re,\ff F)$.

\begin{thm}\label{embed}
The category $\cc S$ is triangulated. Moreover, there is a
contravariant equivalence of triangulated categories
   $$T:D(\Re,\ff F)\to\cc S.$$
\end{thm}

\begin{proof}
By~\cite{Gar1} the natural functor
   $$j:\Re\to D(\Re,\ff F)$$
is a universal excisive homotopy invariant homology theory. Consider
the homology theory
   $$t:\Re\to\shs(\Re)^{\op}$$
that takes an algebra $A\in\Re$ to $\Sigma^\infty rA$. It is
homotopy invariant and excisive, hence there is a unique
triangulated functor
   $$T:D(\Re,\ff F)\to\shs(\Re)^{\op}$$
such that $t=T\circ j$. If we apply $T$ to the loop extension
   $$\Omega A\to EA\to A,$$
we get an isomorphism
   $$T(\Omega A)\cong \Sigma^\infty rA[1],$$
which is functorial in $A$.

It follows from Comparison Theorem~B of~\cite{Gark} and
Corollary~\ref{rbra} that $T$ is full and faithful. Every object of
$\cc S$ is plainly equivalent to the image of an object in
$D(\Re,\ff F)$.
\end{proof}

\begin{rem}{\rm
Suppose $I$ is an infinite index set and $\{B_i\}_{i\in I}$ is a
family of algebras from $\Re$ such that the algebra $B=\oplus_IB_i$
is in $\Re$. Then $\Sigma^\infty rB$ is a compact object of
$\shs(\Re)$, but $\oplus_I\Sigma^\infty r(B_i)$ may not be compact.
Furthermore, suppose $B=\oplus_IB_i$ is also a direct sum object of
the $B_i$-s in the triangulated category $D(\Re,\ff F)$. Then
$\Hom_{D(\Re,\ff F)}(B,\oplus_IC_i)\ne\oplus_I\Hom_{D(\Re,\ff
F)}(B,C_i)$ in general, where $\{C_i\}_{i\in I}$ is a family of
algebras from $\Re$ such that the algebra $\oplus_IC_i$ is in $\Re$.

For instance, consider the triangulated category $KK$ of
Kasparov~\cite{Kas}, with which $D(\Re,\ff F)$ shares many
properties. It follows from~\cite[1.12]{RS} that $KK$ has countable
coproducts given by $A=\oplus_IA_i$, where $I$ is a countable set.
However, the functor $KK(A,-)$ does not respect countable coproducts
by~\cite[7.12]{RS}.

}\end{rem}

Recall from~\cite{Gark} that we can vary $\Re$ in the following
sense. If $\Re'$ is another $T$-closed admissible category of
algebras containing $\Re$, then $D(\Re,\ff F)$ is a full subcategory
of $D(\Re',\ff F)$.

\section{Morita stable algebraic Kasparov $K$-theory}

If $A$ is an algebra and $n > 0$ is a positive integer, then there
is a natural inclusion $\iota: A\to M_nA$ of algebras, sending $A$
to the upper left corner of $M_nA$. Throughout this section $\Re$ is
a small $T$-closed admissible category of $k$-algebras with $M_n
A\in\Re$ for every $A\in\Re$ and $n\geq 1$.

Denote by $U_\bullet\Re_{I,J}^{mor}$ the model category obtained
from $U_\bullet\Re_{I,J}$ by Bousfield localization with respect to
the family of maps of cofibrant objects
   $$\{r(M_nA)\to rA\mid A\in\Re,n>0\}.$$

Let $Sp_{mor}(\Re)$ be the stable model category of $S^1$-spectra
associated with $U_\bullet\Re_{I,J}^{mor}$. Observe that it is also
obtained from $Sp(\Re)$ by Bousfield localization with respect to
the family of maps of cofibrant objects in $Sp(\Re)$
   $$\{F_s(r(M_nA))\to F_s(rA)\mid A\in\Re,n>0,s\geq 0\}.$$
Here $F_s:U_\bullet\Re_{I,J}^{mor}\to Sp_{mor}(\Re)$ is the
canonical functor adjoint to the evaluation functor
$Ev_s:Sp_{mor}(\Re)\to U_\bullet\Re_{I,J}^{mor}$.

\begin{defs}{\rm
(see~\cite{Gark}) (1) The {\it Morita stable algebraic Kasparov
$K$-theory space\/} of two algebras $A,B\in\Re$ is the space
   $$\cc K^{mor}(A,B)=\lp(\cc K(A,B)\to\cc K(A,M_2k\otimes B)\to\cc K(A,M_3k\otimes B)\to\cdots).$$
Its homotopy groups will be denoted by $\cc K^{mor}_n(A,B)$, $n\geq
0$.

(2) A functor $X:\Re\to\bb S/(Spectra)$ is {\it Morita invariant\/}
if each morphism $X(A)\to X(M_nA)$, $A\in\Re,n>0$, is a weak
equivalence.

(3) An excisive, homotopy invariant homology theory $X:\Re\to\cc T$
is {\it Morita invariant\/} if each morphism $X(A)\to X(M_nA)$,
$A\in\Re, n>0$, is an isomorphism.

(4) The {\it Morita stable algebraic Kasparov $K$-theory spectrum\/}
of $A,B\in\Re$ is the $\Omega$-spectrum
   $$\mathbb{K}^{mor}(A,B)=(\mathcal{K}^{mor}(A,B),\mathcal{K}^{mor}(JA,B),\mathcal{K}^{mor}(J^2A,B),\ldots).$$

}\end{defs}

Denote by $\shs^{mor}(\Re)$ the (stable) homotopy category of
$Sp_{mor}(\Re)$. It is a compactly generated triangulated category
with compact generators $\{\Sigma^\infty rA[n]\}_{A\in\Re,n\in\bb
Z}$. Let $\cc S_{mor}$ be the full subcategory of $\shs^{mor}(\Re)$
whose objects are $\{\Sigma^\infty rA[n]\}_{A\in\Re,n\in\bb Z}$.

Recall from~\cite{Gar1} the definition of the triangulated category
$D_{mor}(\Re,\ff F)$. Its objects are those of $\Re$ and the set of
morphisms between two algebras $A,B\in\Re$ is defined as the colimit
of the sequence of abelian groups
   $$D(\Re,\ff F)(A,B)\to D(\Re,\ff F)(A,M_2B)\to D(\Re,\ff F)(A,M_3B)\to\cdots$$
There is a canonical functor $\Re\to D_{mor}(\Re,\ff F)$. It is a
universal excisive, homotopy invariant and Morita invariant homology
theory on $\Re$.

\begin{thm}\label{mainresmor}
Given $A\in\Re$ the composite map
   \begin{equation}\label{stmor1}
    \Sigma^\infty rA\bl i\to\cc R(A)\bl j\to\bb K(A,-)\to\bb K^{mor}(A,-)
   \end{equation}
is a stable equivalence in $Sp_{mor}(\Re)$, functorial in $A$. In
particular, there is a natural isomorphism
   $$\shs^{mor}(\Re)(\Sigma^\infty rB[n],\Sigma^\infty rA)\cong\bb K_n^{mor}(A,B)$$
for all $A,B\in\Re$ and $n\in\bb Z$. Furthermore, the category $\cc
S_{mor}$ is triangulated and there is a contravariant equivalence of
triangulated categories
   $$T:D_{mor}(\Re,\ff F)\to\cc S_{mor}.$$
\end{thm}

\begin{proof}
Let $\cc S^c$ and $\cc S^c_{mor}$ be the categories of compact
objects in $\shs(\Re)$ and $\shs^{mor}(\Re)$ respectively. Denote by
$\cc R$ the full triangulated subcategory of $\cc S$ generated by
objects
   $$\{\cone(\Sigma^\infty r(M_nA)\to\Sigma^\infty rA)[k]\mid A\in\Re,n>0,k\in\bb Z\}.$$
Let $\cc R^c$ be the thick closure of $\cc R$ in $\shs(\Re)$. It
follows from~\cite[2.1]{N} that the natural functor
   $$\cc S^c/\cc R^c\to\cc S^c_{mor}$$
is full and faithful and $\cc S^c_{mor}$ is the thick closure of
$\cc S^c/\cc R^c$.

We claim that the natural functor
   \begin{equation}\label{sr}
    \cc S/\cc R\to\cc S^c/\cc R^c
   \end{equation}
is full and faithful. For this consider a map $\alpha:X\to Y$ in
$\cc S^c$ such that its cone $Z$ is in $\cc R^c$ and $Y\in\cc S$. We
can find $Z'\in\cc R^c$ such that $Z\oplus Z'$ is isomorphic to an
object $W\in\cc R$. Construct a commutative diagram in $\cc S^c$
   $$\xymatrix{U\ar[r]\ar[d]_s&Y\ar[r]\ar@{=}[d]&W\ar[r]\ar[d]_p&\Sigma U\ar[d]\\
               X\ar[r]^\alpha&Y\ar[r]&Z\ar[r]&\Sigma X,}$$
where $p$ is the natural projection. We see that $\alpha s$ is such
that its cone $W$ belongs to $\cc R$. Standard facts for
Gabriel--Zisman localization theory imply~\eqref{sr} is a fully
faithful embedding. It also follows that
   $$\cc S_{mor}=\cc S/\cc R.$$

We want to compute Hom sets in $\cc S/\cc R$. For this observe first
that there is a contravariant equivalence of triangulated categories
   $$\tau:D(\Re,\ff F)/\ff U\to\cc S_{mor},$$
where $\ff U$ is the smallest full triangulated subcategory of
$D(\Re,\ff F)$ containing
   $$\{\cone( A\bl\iota\to M_nA)\mid A\in\Re,n>0\}.$$
This follows from Theorem~\ref{embed}.

By construction, every excisive homotopy invariant Morita invariant
homology theory $\Re\to\cc T$ factors through $D(\Re,\ff F)/\ff U$.
Since $\Re\to D_{mor}(\Re,\ff F)$ is a universal excisive homotopy
invariant Morita invariant homology theory~\cite{Gar1}, we see that
there exists a triangle equivalence of triangulated categories
   $$D_{mor}(\Re,\ff F)\simeq D(\Re,\ff F)/\ff U.$$
So there is a natural contravariant triangle equivalence of
triangulated categories
   $$T:D_{mor}(\Re,\ff F)\to\cc S_{mor}.$$
Using this and~\cite[9.8]{Gark}, there is a natural isomorphism
   $$\cc S_{mor}(\Sigma^\infty rB[n],\Sigma^\infty rA)\cong\bb K_n^{mor}(A,B)$$
for all $A,B\in\Re$ and $n\in\bb Z$. The fact that~\eqref{stmor1} is
a stable equivalence in $Sp_{mor}(\Re)$ is now obvious.
\end{proof}

\section{Stable algebraic Kasparov $K$-theory}

If $A$ is an algebra set $M_\infty A=\cup_n M_nA$. There is a
natural inclusion $\iota: A\to M_\infty A$ of algebras, sending $A$
to the upper left corner of $M_\infty A$. Throughout the section
$\Re$ is a small $T$-closed admissible category of $k$-algebras with
$M_\infty(A)\in\Re$ for all $A\in\Re$.

Denote by $U_\bullet\Re_{I,J}^{\infty}$ the model category obtained
from $U_\bullet\Re_{I,J}$ by Bousfield localization with respect to
the family of maps of cofibrant objects
   $$\{r(M_\infty A)\to rA\mid A\in\Re\}.$$

Let $Sp_{\infty}(\Re)$ be the stable model category of $S^1$-spectra
associated with $U_\bullet\Re_{I,J}^{\infty}$. Observe that it is
also obtained from $Sp(\Re)$ by Bousfield localization with respect
to the family of maps of cofibrant objects in $Sp(\Re)$
   $$\{F_s(r(M_\infty A))\to F_s(rA)\mid A\in\Re,s\geq 0\}.$$

\begin{defs}{\rm
(see~\cite{Gark}) (1) The {\it stable algebraic Kasparov $K$-theory
space\/} of two algebras $A,B\in\Re$ is the space
   $$\cc K^{st}(A,B)=\lp(\cc K(A,B)\to\cc K(A,M_\infty k\otimes B)\to\cc K(A,M_\infty k\otimes M_\infty k\otimes  B)\to\cdots).$$
Its homotopy groups will be denoted by $\cc K^{st}_n(A,B)$, $n\geq
0$.

(2) A functor $X:\Re\to\bb S/(Spectra)$ is {\it stable\/} or {\it
$M_\infty$-invariant\/} if $X(A)\to X(M_\infty A)$ is a weak
equivalence for all $A\in\Re$.

(3) An excisive, homotopy invariant homology theory $X:\Re\to\cc T$
is {\it stable\/} or {\it $M_\infty$-invariant\/} if $X(A)\to
X(M_\infty A)$ is an isomorphism for all $A\in\Re$.

(4) The {\it stable algebraic Kasparov $K$-theory spectrum\/} for
$A,B\in\Re$ is the $\Omega$-spectrum
   $$\mathbb{K}^{st}(A,B)=(\mathcal{K}^{st}(A,B),\mathcal{K}^{st}(JA,B),\mathcal{K}^{st}(J^2A,B),\ldots).$$

}\end{defs}

Denote by $\shs^{\infty}(\Re)$ the (stable) homotopy category of
$Sp_{\infty}(\Re)$. It is a compactly generated triangulated
category with compact generators $\{\Sigma^\infty
rA[n]\}_{A\in\Re,n\in\bb Z}$. Let $\cc S_{\infty}$ be the full
subcategory of $\shs^{\infty}(\Re)$ whose objects are
$\{\Sigma^\infty rA[n]\}_{A\in\Re,n\in\bb Z}$.

Recall from~\cite{Gar1} the definition of the triangulated category
$D_{st}(\Re,\ff F)$. Its objects are those of $\Re$ and the set of
morphisms between two algebras $A,B\in\Re$ is defined as the colimit
of the sequence of abelian groups
   $$D(\Re,\ff F)(A,B)\to D(\Re,\ff F)(A,M_\infty k\otimes_kB)
     \to D(\Re,\ff F)(A,M_\infty k\otimes_kM_\infty k\otimes_kB)\to\cdots$$
There is a canonical functor $\Re\to D_{st}(\Re,\ff F)$. It is the
universal excisive, homotopy invariant and stable homology theory on
$\Re$.

The proof of the next result literally repeats that of
Theorem~\ref{mainresmor} if we replace the algebras $M_nA$ with
$M_\infty A$ and the categories $\cc S_{mor}$ and $D_{mor}(\Re,\ff
F)$ with $\cc S_\infty$ and $D_{st}(\Re,\ff F)$ respectively.

\begin{thm}\label{mainresst}
Given $A\in\Re$ the composite map
   \begin{equation*}\label{stmor}
    \Sigma^\infty rA\bl i\to\cc R(A)\bl j\to\bb K(A,-)\to\bb K^{st}(A,-)
   \end{equation*}
is a stable equivalence in $Sp_{\infty}(\Re)$, functorial in $A$. In
particular, there is a natural isomorphism
   $$\shs^{\infty}(\Re)(\Sigma^\infty rB[n],\Sigma^\infty rA)\cong\bb K_n^{st}(A,B)$$
for all $A,B\in\Re$ and $n\in\bb Z$. Furthermore, the category $\cc
S_{\infty}$ is triangulated and there is a contravariant equivalence
of triangulated categories
   $$T:D_{st}(\Re,\ff F)\to\cc S_{\infty}.$$
\end{thm}

Let $\Gamma A$, for $A\in\ahaw$, be the algebra of $\bb N\times\bb
N$-matrices which satisfy the following two properties.
\begin{itemize}
\item[(i)] The set $\{a_{ij}\mid i,j \in \bb N\}$ is finite.
\item[(ii)] There exists a natural number $N \in \bb N$ such that each row and each column has at most
$N$ nonzero entries.
\end{itemize}
$M_\infty A\subset\Gamma A$ is an ideal. We put
   $$\Sigma A=\Gamma A/M_\infty A.$$
We note that $\Gamma A$, $\Sigma A$ are the cone and suspension
rings of $A$ considered by Karoubi and Villamayor in \cite[p.
269]{KV}, where a different but equivalent definition is given.
By~\cite{CT} there are natural ring isomorphisms
   $$\Gamma A\cong\Gamma k\otimes A,\quad \Sigma A\cong\Sigma k\otimes A.$$
We call the short exact sequence
   $$M_\infty A\rightarrowtail\Gamma A\twoheadrightarrow\Sigma A$$
the {\it cone extension}. By~\cite{CT} $\Gamma
A\twoheadrightarrow\Sigma A$ is a split surjection of $k$-modules.

Let $\tau$ be the $k$-algebra which is unital and free on two
generators $\alpha$ and $\beta$ satisfying the relation $\alpha
\beta=1$. By~\cite[4.10.1]{CT} the kernel of the natural map
   $$\tau\to k[t^{\pm 1}]$$
is isomorphic to $M_\infty k$. We set $\tau_0 = \tau
\oplus_{k[t^{\pm 1}]} \sigma$.

Let $A$ be a $k$-algebra. We get an extension
\[ \xymatrix{ M_\infty A  \ar[r] & \tau A \ar[r] & A[t^{\pm 1}],} \]
and an analogous extension
   \begin{equation*}\label{taunol}
    \xymatrix{ M_\infty A  \ar[r] & \tau_0 A \ar[r] & {\sigma} A.}
   \end{equation*}

\begin{defs}{\rm
We say that an admissible category of $k$-algebras $\Re$ is {\it
$\tau_0$-closed\/} (respectively {\it $\Gamma$-closed}) if
$\tau_0A\in\Re$ (respectively $\Gamma A\in\Re$) for all $A\in\Re$.

}\end{defs}

Cuntz~\cite{Cu2,Cu,Cu1} constructed a triangulated category
$kk^{lca}$ whose objects are the locally convex algebras. Later
Corti\~nas--Thom~\cite{CT} constructed in a similar fashion a
triangulated category $kk$ whose objects are all $k$-algebras
$\ahaw$. If we suppose that $\Re$ is also $\Gamma$-closed, then one
can define a full triangulated subcategory $kk(\Re)$ of $kk$ whose
objects are those of $\Re$.

It can be shown similar to~~\cite[7.4]{Gar} or~\cite[9.4]{Gar1} that
there is an equivalence of triangulated categories
   $$D_{st}(\Re,\ff F)\lra{\sim} kk(\Re).$$

An important computational result of Corti\~nas--Thom~\cite{CT}
states that there is an isomorphism of graded abelian groups
   $$\bigoplus_{n\in\bb Z}kk(\Re)(k,\Omega^nA)\cong\bigoplus_{n\in\bb Z}KH_n(A),$$
where the right hand side is the homotopy $K$-theory of $A\in\Re$ in
the sense of Weibel~\cite{W1}.

Summarizing the above arguments together with
Theorem~\ref{mainresst} we obtain the following

\begin{thm}\label{mainresstcu}
Suppose $\Re$ is $\Gamma$-closed. Then there is a contravariant
equivalence of triangulated categories
   $$kk(\Re)\to\cc S_{\infty}.$$
Moreover, there is a natural isomorphism
   $$\shs^{\infty}(\Re)(\Sigma^\infty rA[n],\Sigma^\infty r(k))\cong KH_n(A).$$
for any $A\in\Re$ and any integer $n$.
\end{thm}

\section{$K$-motives of algebras}

Throughout the section we assume that $\Re$ is a small tensor closed
and $T$-closed admissible category of $k$-algebras with
$M_\infty(k)\in\Re$. It follows that $M_\infty A:\cong A\otimes
M_\infty(k)\in\Re$ for all $A\in\Re$.

In this section we define and study the triangulated category of
$K$-motives. It shares many properties with the category of
$K$-motives for algebraic varieties constructed in~\cite{GP,GP1}.

Since $\Re$ is tensor closed, it follows that
$U_\bullet\Re_{I,J}^{\infty}$ is a monoidal model category. Let
$Sp_\infty^\Sigma(\Re)$ be the monoidal category of symmetric
spectra in the sense of Hovey~\cite{H} associated to
$U_\bullet\Re_{I,J}^{\infty}$.

\begin{defs}{\rm
The {\it category of $K$-motives $DK(\Re)$} is the stable homotopy
category of $Sp_\infty^\Sigma(\Re)$. The {\it $K$-motive\/} $M_K(A)$
of an algebra $A\in\Re$ is the image of $A$ in $DK(\Re)$, that is
$M_K(A)=\Sigma^\infty rA$. Thus one has a canonical contravariant
functor
   $$M_K:\Re\to DK(\Re)$$
sending algebras to their $K$-motives.

}\end{defs}

The following proposition follows from standard facts for monoidal
model categories.

\begin{prop}\label{kmo}
$DK(\Re)$ is a symmetric monoidal compactly generated triangulated
category with compact generators $\{M_K(A)\}_{A\in\Re}$. For any two
algebras $A,B\in\Re$ one has a natural isomorphism
   $$M_K(A)\otimes M_K(B)\cong M_K(A\otimes B).$$
Furthermore, any extension of algebras in $\Re$
   $$(E):\quad A\to B\to C$$
induces a triangle in $DK(\Re)$
   $$M_K(E):\quad M_K(C)\to M_K(B)\to M_K(A)\xrightarrow+.$$
\end{prop}

There is a pair of adjoint functors
    $$V:Sp_\infty(\Re)\leftrightarrows Sp_\infty^\Sigma(\Re):U,$$
where $U$ is the right Quillen forgetful functor. These form a
Quillen equivalence. In particular, the induced functors
   $$V:\shs^\infty(\Re)\leftrightarrows DK(\Re):U$$
are equivalences of triangulated categories. It follows from
Proposition~\ref{kmo} that $\shs^{\infty}(\Re)$ is a symmetric
monoidal category and
   $$\Sigma^\infty rA\otimes\Sigma^\infty rB\cong\Sigma^\infty r(A\otimes B)$$
for all $A,B\in\Re$. Moreover,
   $$V(\Sigma^\infty rA)\cong M_K(A)$$
for all $A\in\Re$.

Summarizing the above arguments together with
Theorem~\ref{mainresst} we get the following

\begin{thm}\label{kmotives}
For any two algebras $A,B\in\Re$ and any integer $n$ one has a
natural isomorphism of abelian groups
   \begin{equation*}
    DK(\Re)(M_K(B)[n],M_K(A))\cong\bb K^{st}_n(A,B).
   \end{equation*}
If $\cc T$ is the full subcategory of $DK(\Re)$ spanned by
$K$-motives of algebras $\{M_K(A)\}_{A\in\Re}$ then $\cc T$ is
triangulated and there is an equivalence of triangulated categories
   $$D_{st}(\Re,\ff F)\to\cc T^{\op}$$
sending an algebra $A\in\Re$ to its $K$-motive $M_K(A)$.
\end{thm}

The next result is reminiscent of a similar result for $K$-motives
of algebraic varieties in the sense of~\cite{GP,GP1} identifying the
$K$-motive of the point with algebraic $K$-theory.

\begin{cor}\label{kak}
Suppose $\Re$ is $\Gamma$-closed. Then for any algebra $A$ and any
integer $n$ one has a natural isomorphism of abelian groups
   \begin{equation*}
    DK(\Re)(M_K(A)[n],M_K(k))\cong KH_n(A),
   \end{equation*}
where the right hand side is the $n$-th homotopy $K$-theory group in
the sense of Weibel~\cite{W1}.
\end{cor}

\begin{proof}
This follows from Theorem~\cite[10.6]{Gar1} and the preceding
theorem.
\end{proof}

We finish the section by showing that the category $kk(\Re)$ of
Corti\~nas--Thom~\cite{CT} can be identified with the $K$-motives of
algebras.

\begin{thm}\label{kktop}
Suppose $\Re$ is $\Gamma$-closed. Then there is a natural
equivalence of triangulated categories
   $$kk(\Re)\lra{\sim}\cc T^{\op}$$
sending an algebra $A\in\Re$ to its $K$-motive $M_K(A)$.
\end{thm}

\begin{proof}
This follows from Theorem~\ref{kmotives} and the fact that
$D_{st}(\Re,\ff F)$ and $kk(\Re)$ are triangle equivalent
(see~\cite[7.4]{Gar} or~\cite[9.4]{Gar1}).
\end{proof}

The latter theorem shows in particular that $kk(\Re)$ is embedded
into the compactly generated triangulated category of $K$-motives
$DK(\Re)$ and generates it.

\section{The $\bb G$-stable theory}

The stable motivic homotopy theory over a field is the homotopy
theory of $T$-spectra, where $T=S^1\wedge\bb G_m$
(see~\cite{VoeICM,J}). There are various equivalent definitions of
the theory, one of which is given in terms of $(S^1,\bb
G_m)$-bispectra. In our context the role of the motivic space $\bb
G_m$ is played by $\sigma=(t-1)k[t^{\pm 1}]$. Its simplicial functor
$r(\sigma)$ is denoted by $\bb G$. In this section we define the
stable category of $(S^1,\bb G)$-bispectra and construct an explicit
fibrant replacement of the $(S^1,\bb G)$-bispectrum
$\Sigma^\infty_{\bb G}\Sigma^\infty rA$ of an algebra $A$. One can
also define a Quillen equivalent category of $T$-spectra, where
$T=S^1\wedge\bb G$, and compute an explicit fibrant replacement for
the $T$-spectrum of an algebra. However we prefer to work with
$(S^1,\bb G)$-bispectra rather than $T$-spectra in order to study
$K$-motives of algebras in terms of associated $(S^1,\bb
G)$-bispectra (see the next section).

Throughout the section we assume that $\Re$ is a small tensor closed
and $T$-closed admissible category of $k$-algebras. We have that
$\sigma A:=A\otimes\sigma\in\Re$ for all $A\in\Re$.

Recall that $U_\bullet\Re_{I,J}$ is a monoidal model category. It
follows from~\cite[6.3]{H} that $Sp(\Re)$ is a
$U_\bullet\Re_{I,J}$-model category. In particular
   $$-\otimes\bb G:Sp(\Re)\to Sp(\Re)$$
is a left Quillen endofunctor.

By definition, a {\it $(S^1,\bb G)$-bispectrum\/} or {\it
bispectrum\/} $\cc E$ is given by a sequence $(E_0,E_1,\ldots)$,
where each $E_j$ is a $S^1$-spectrum of $Sp(\Re)$, together with
bonding morphisms $\epsilon_n:E_n\wedge\bb G\to E_{n+1}$. Maps are
sequences of maps in $Sp(\Re)$ respecting the bonding morphisms. We
denote the category of bispectra by $Sp_{\bb G}(\Re)$. It can be
regarded as the category of $\bb G$-spectra on $Sp(\Re)$ in the
sense of Hovey~\cite{H}.

$Sp_{\bb G}(\Re)$ is equipped with the stable
$U_\bullet\Re_{I,J}$-model structure in which weak equivalences are
defined by means of bigraded homotopy groups. The bispectrum object
$\cc E$ determines a sequence of maps of $S^1$-spectra
   $$E_0\xrightarrow{\tilde\epsilon_0}\Omega_{\bb G}E_1
     \xrightarrow{\Omega_{\bb G}(\tilde\epsilon_1)}\Omega_{\bb G}^2E_2\to\cdots$$
where $\Omega_{\bb G}$ is the functor $\underline{\Hom}(\bb G,-)$
and $\tilde\epsilon_n$-s are adjoint to the structure maps of $\cc
E$. We define $\pi_{p,q}\cc E$ in $A$-sections as the colimit
   $$\colim_l\bigl(\Hom_{\shs(\Re)}(S^{p-q},\Omega_{\bb G}^{q+l}JE_l(A))\to\Hom_{\shs(\Re)}(S^{p-q},\Omega_{\bb G}^{q+l+1}JE_{l+1}(A))\to\cdots\bigr)$$
once $\cc E$ has been replaced up to levelwise equivalence by a
levelwise fibrant object $J\cc E$ so that the ``loop" constructions
make sense. We also call $\pi_{*,q}\cc E$ the homotopy groups of
weight $q$.

By definition, a map of bispectra is a weak equivalence in $Sp_{\bb
G}(\Re)$ if it induces an isomorphism on bigraded homotopy groups.
We denote the homotopy category of $Sp_{\bb G}(\Re)$ by
$\shsg(\Re)$. It is a compactly generated triangulated category.

In order to define the main $(S^1,\bb G)$-bispectrum of this
section, denoted by $\bb{KG}(A,-)$, we should first establish some
facts for algebra homomorphisms.

Suppose $A,C\in\Re$, then one has a commutative diagram
   $$\xymatrix{J(A\otimes C)\ar@{ >->}[r]\ar[d]_{\gamma_{A,C}}&T(A\otimes C)\ar@{->>}[r]^(.55){\eta_{A\otimes C}}\ar[d]&A\otimes C\ar@{=}[d]\\
               JA\otimes C\ar@{ >->}[r]&T(A)\otimes C\ar@{->>}[r]^(.58){\eta_A\otimes C}&A\otimes C}$$
in which $\gamma_{A,C}$ is uniquely determined by the split
monomorphism $i_A\otimes C:A\otimes C\to T(A)\otimes C$.

One sets $\gamma_{A,C}^0:=1_{A\otimes C}$. We construct inductively
   $$\gamma_{A,C}^n:J^n(A\otimes C)\to J^n(A)\otimes C,\quad n\geq 1.$$
Namely, $\gamma_{A,C}^{n+1}$ is the composite
   $$J^{n+1}(A\otimes C)\xrightarrow{J(\gamma_{A,C}^n)}J(J^n(A)\otimes C)\xrightarrow{\gamma_{J^nA,C}}J^{n+1}(A)\otimes C.$$
Given $n\geq 0$ we define a map
   $$t_n=t_n^{A,C}:\cc K(J^nA,-)\to\cc K(J^{n}(A\otimes C),-\otimes C)=\underline{\Hom}(rC,\cc K(J^{n}(A\otimes C),-))$$
as follows. Let $B\in\Re$ and $(\alpha:J^{n+m}A\to\bb
B(\Omega^m))\in\cc K(J^nA,B)$. We set $t_n(\alpha)\in\cc
K(J^{n}(A\otimes C),B\otimes C)$ to be the composite
   $$J^{n+m}(A\otimes C)\xrightarrow{\gamma_{A,C}^{n+m}}J^{n+m}(A)\otimes C\xrightarrow{\alpha\otimes C}\bb B^\Delta(\Omega^m)\otimes C
     \bl\tau\cong(\bb B\otimes\bb C)^\Delta(\Omega^m).$$
Here $\tau$ is a canonical isomorphism (see~\cite[3.1.3]{CT}) and
$(\bb B\otimes\bb C)^\Delta$ stands for the simplicial ind-algebra
   $$[m,\ell]\mapsto\Hom_{\bb S}(\sd^m\Delta^\ell,(B\otimes C)^\Delta)=(B\otimes C)^{\sd^m\Delta^\ell}\cong k^{\sd^m\Delta^\ell}\otimes(B\otimes C).$$
One has to verify that $t_n$ is consistent with maps
   $$\Hom_{\inda}(J^{n+m}A,\bb B^\Delta(\Omega^m))\xrightarrow{\varsigma}\Hom_{\inda}(J^{n+m+1}A,\bb B^\Delta(\Omega^{m+1})).$$
More precisely, we must show that the map\footnotesize
   $$J^{n+m+1}(A\otimes C)\xrightarrow{J(\gamma_{A,C}^{n+m})}J(J^{n+m}A\otimes C)\xrightarrow{J(\alpha\otimes 1)}J(\bb B^\Delta(\Omega^m)\otimes
   C)\bl{J\tau}\cong J((\bb B\otimes\bb C)^\Delta(\Omega^m))\xrightarrow{\xi_\upsilon}(\bb B\otimes\bb C)^\Delta(\Omega^{m+1})$$
\normalsize is equal to the map\footnotesize
   $$J^{n+m+1}(A\otimes C)\xrightarrow{\gamma_{A,C}^{n+m+1}}J^{n+m+1}A\otimes C\xrightarrow{J\alpha\otimes 1}J(\bb B^\Delta(\Omega^{m}))\otimes C
     \xrightarrow{\xi_\upsilon\otimes 1}\bb B^\Delta(\Omega^{m+1})\otimes C\bl\tau\cong(\bb B\otimes\bb C)^\Delta(\Omega^{m+1}).$$
\normalsize The desired property follows from commutativity of the
diagram (we use here~\cite[3.4]{Gark})
   $$\xymatrix{
   J^{n+m+1}(A\otimes C)\ar[d]_{J(\gamma_{A,C}^{n+m})}\\
   J(J^{n+m}A\otimes C)\ar[r]^{\gamma_{J^{n+m}A,C}}\ar[d]_{J(\alpha\otimes 1)}&J^{n+m+1}A\otimes C\ar@{ >->}[r]\ar[d]_{J(\alpha)\otimes 1}&TJ^{n+m}A\otimes C\ar@{->>}[r]\ar[d]&J^{n+m}A\otimes C\ar[d]^{\alpha\otimes 1}\\
   J(\bb B^\Delta(\Omega^{m})\otimes C)\ar@{=}[d]\ar[r]&J(\bb B^\Delta(\Omega^{m}))\otimes C\ar[d]_{\xi_\upsilon\otimes 1}\ar@{ >->}[r]&T(\bb B^\Delta(\Omega^{m}))\otimes C\ar[d]\ar@{->>}[r]&\bb B^\Delta(\Omega^{m})\otimes C\ar@{=}[d]\\
   J(\bb B^\Delta(\Omega^{m})\otimes C)\ar[d]_{J\tau}\ar[r]&\bb B^\Delta(\Omega^{m+1})\otimes C\ar@{ >->}[r]\ar[d]_\tau&P(\bb B^\Delta(\Omega^{m}))\otimes C\ar@{->>}[r]\ar[d]^\tau&\bb B^\Delta(\Omega^{m})\otimes C\ar[d]^\tau\\
   J((\bb B\otimes\bb C)^\Delta(\Omega^m))\ar[r]_{\xi_\upsilon}&(\bb B\otimes\bb C)^\Delta(\Omega^{m+1})\ar@{ >->}[r]&P(\bb B\otimes\bb C)^\Delta(\Omega^m)\ar@{->>}[r]&(\bb B\otimes\bb C)^\Delta(\Omega^m)}$$
We see that $t_n$ is well defined. We claim that the collection of
maps $(t_n)_n$ defines a map of $S^1$-spectra
   $$t:\bb K(A,B)\to\bb K(A\otimes C,B\otimes C).$$
We have to check that for each $n\geq 0$ the diagram
   $$\xymatrix{\cc K(J^nA,B)\ar[r]^\cong\ar[d]_{t_n}&\Omega\cc K(J^{n+1}A,B)\ar[d]^{\Omega t_{n+1}}\\
               \cc K(J^n(A\otimes C),B\otimes C)\ar[r]^(.47)\cong&\Omega\cc K(J^{n+1}(A\otimes C),B\otimes C)}$$
is commutative. But this directly follows from the definition of the
horizontal maps (see~\cite[5.1]{Gark}) and arguments above made for
$t_n$-s.

If we replace $C$ by $\sigma$ we get that the array
\begin{equation*}
\xymatrix{
&\vdots & \vdots & \vdots\\
\bb K(\sigma^2 A,B):&\cc K(\sigma^2 A,B) & \cc K(J\sigma^2 A,B) & \cc K(J^2\sigma^2 A,B) & \cdots\\
\bb K(\sigma A,B): &\cc K(\sigma A,B)& \cc K(J\sigma A,B) & \cc K(J^2\sigma A,B) & \cdots\\
\bb K(A,B):&\cc K(A,B) & \cc K(JA,B) & \cc K(J^2A,B) & \cdots }
\end{equation*}
together with structure maps
   $$\bb K(\sigma^n A,-)\otimes\bb G\to\bb K(\sigma^{n+1}A,-)$$
defined as adjoint maps to
   $$t:\bb K(\sigma^n A,-)\to\underline{\Hom}(\bb G,\bb K(\sigma^{n+1}A,-))$$
forms a $(S^1,\bb G)$-bispectrum, which we denote by $\bb{KG}(A,-)$.

There is a natural map of $(S^1,\bb G)$-bispectra
   $$\Gamma:\Sigma_{\bb G}^\infty\Sigma^\infty rA\to\bb{KG}(A,-),$$
where $\Sigma_{\bb G}^\infty\Sigma^\infty rA$ is the $(S^1,\bb
G)$-bispectrum represented by the array
\begin{equation*}
\xymatrix{
&\vdots & \vdots \\
\Sigma^\infty rA\otimes\bb G^2:&rA\otimes\bb G^2\ (\cong r(\sigma^2A)) &(rA\wedge S^1)\otimes\bb G^2\ (\cong r(\sigma^2A)\wedge S^1)& \cdots\\
\Sigma^\infty rA\otimes\bb G:&rA\otimes\bb G\ (\cong r(\sigma A))&(rA\wedge S^1)\otimes\bb G\ (\cong r(\sigma A)\wedge S^1)& \cdots\\
\Sigma^\infty rA:&rA & rA\wedge S^1 &\cdots }
\end{equation*}
with obvious structure maps.

By Theorem~\ref{mainres} each map
   $$\Gamma_n:\Sigma^\infty rA\otimes\bb G^n\to\bb{KG}(A,-)_n=\bb K(\sigma^nA,-)$$
is a stable weak equivalence in $Sp(\Re)$. By~\cite{Gark} each $\bb
K(\sigma^nA,-)$ is a fibrant object in $Sp(\Re)$. For each $n\geq 0$
we set
   $$\Theta^\infty_{\bb G}\bb{KG}(A,-)_n=\colim(\bb K(\sigma^nA,-)\xrightarrow{t_0}\bb K(\sigma^{n+1}A,-\otimes\sigma)
     \xrightarrow{\Omega_{\bb G}(t_1)}\bb K(\sigma^{n+2}A,-\otimes\sigma^2)\to\cdots).$$
By specializing a collection of results in~\cite[section~4]{H} to
our setting we have that $\Theta^\infty_{\bb G}\bb{KG}(A,-)$ is a
fibrant bispectrum and the natural map
   $$j:\bb{KG}(A,-)\to\Theta^\infty_{\bb G}\bb{KG}(A,-)$$
is a weak equivalence in $Sp_{\bb G}(\Re)$.

We have thus shown that $\Theta^\infty_{\bb G}\bb{KG}(A,-)$ is an
explicit fibrant replacement for the bispectrum $\Sigma_{\bb
G}^\infty\Sigma^\infty rA$ of the algebra $A$. Denote by $\cc
K^\sigma(A,B)$ the $(0,0)$-space of the bispectrum
$\Theta^\infty_{\bb G}\bb{KG}(A,B)$. It is, by construction, the
colimit
   $$\colim_n\cc K(\sigma^nA,\sigma^n B).$$
Its homotopy groups will be denoted by $\cc K_n^\sigma(A,B)$, $n\geq
0$.

\begin{thm}\label{maingsp}
Let $A$ be an algebra in $\Re$; then the composite map
   $$j\circ\Gamma:\Sigma_{\bb G}^\infty\Sigma^\infty rA\to\Theta^\infty_{\bb G}\bb{KG}(A,-)$$
is a fibrant replacement of $\Sigma_{\bb G}^\infty\Sigma^\infty rA$.
In particular,
   $$\shsg(\Sigma_{\bb G}^\infty\Sigma^\infty rB,\Sigma_{\bb G}^\infty\Sigma^\infty rA)=
     \cc K^\sigma_0(A,B)$$ for all $B\in\Re$.
\end{thm}

\begin{rem}{\rm
Let $SH(F)$ be the motivic stable homotopy category over a field
$F$. The category $\shsg(\Re)$ shares many properties with $SH(F)$.
The author and Panin~\cite{GP3} have recently computed a fibrant
replacement of $\Sigma^\infty_{s,t}X_+$, $X\in Sm/F$, by developing
the machinery of framed motives. The machinery is based on the
theory of framed correspondences developed by Voevodsky~\cite{Voe2}.
In turn, the computation of Theorem~\ref{maingsp} is possible thanks
to the existence of universal extensions of algebras.

}\end{rem}

Let $F$ be an algebraically closed field of characteristic zero with
an embedding $F\hookrightarrow\bb C$ and let $SH$ be the stable
homotopy category of ordinary spectra. Let $c:SH\to SH(F)$ be the
functor induced by sending a space to the constant presheaf of
spaces on $Sm/F$. Levine~\cite{Lev} has recently shown that $c$ is
fully faithful. This is in fact implied by a result of
Levine~\cite{Lev} saying that the Betti realization functor in the
sense of Ayoub~\cite{Ay}
   $$Re_{B}:SH(F)\to SH$$
gives an isomorphism
   $$Re_{B*}:\pi_{n,0}\cc S_F(F)\to\pi_n(\cc S)$$
for all $n\in\bb Z$. Here $\cc S_F$ is the motivic sphere spectrum
in $SH(F)$ and $\cc S$ is the classical sphere spectrum in $SH$.
These results use recent developments for the spectral sequence
associated with the slice filtration of the motivic sphere $\cc
S_F$.

All this justifies raising the following questions.

\begin{questions}
$(1)$ Is there an admissible category of commutative algebras $\Re$
over the field of complex numbers $\bb C$ such that the natural
functor
   $$c:SH\to\shsg(\Re),$$
induced by the functor $\bb S\to U\Re$ sending a simplicial set to
the constant simplicial functor on $\Re$, is fully faithful?

$(2)$ Let $\Re$ be an admissible category of commutative $\bb
C$-algebras and let $\cc S_{\bb C}$ be the bispectrum
$\Sigma^\infty_{\bb G}\Sigma^\infty r\bb C$. Is it true that the
homotopy groups of weight zero $\pi_{n,0}\cc S_{\bb C}(\bb C)=\cc
K_n^\sigma(\bb C,\bb C)$, $n\geq 0$, are isomorphic to the stable
homotopy groups $\pi_n(\cc S)$ of the classical sphere spectrum?
\end{questions}

We should also mention that one can define $(S^1,\bb G)$-bispectra
by starting at the monoidal category of symmetric spectra
$Sp^{\Sigma}(\Re)$ associated with the monoidal category
$U_\bullet(\Re)_{I,J}$ and then stabilizing the left Quillen functor
$-\otimes\bb G:Sp^{\Sigma}(\Re)\to Sp^{\Sigma}(\Re)$. One produces a
model category $Sp^\Sigma_{\bb G}(\Re)$ of (usual, non-symmetric)
$\bb G$-spectra in $Sp^\Sigma(\Re)$. Using Hovey's~\cite{H}
notation, one has, by definition, $Sp^\Sigma_{\bb G}(\Re)=Sp^{\bb
N}(Sp^\Sigma(\Re),-\otimes\bb G)$.

There is a Quillen equivalence
   $$V:Sp(\Re)\leftrightarrows Sp^\Sigma(\Re):U$$
as well as a Quillen equivalence
   $$V:Sp_{\bb G}(\Re)\leftrightarrows Sp^\Sigma_{\bb G}(\Re):U,$$
where $U$ is the forgetful functor (see~\cite[5.7]{H}).

If we denote by $\shs^\Sigma(\Re)$ and $SH_{S^1,\bb G}^\Sigma(\Re)$
the homotopy categories of $Sp^\Sigma(\Re)$ and $Sp^\Sigma_{\bb
G}(\Re)$ respectively, then one has equivalences of categories
   $$V:\shs(\Re)\leftrightarrows\shs^\Sigma(\Re):U$$
and
   $$V:SH_{S^1,\bb G}(\Re)\leftrightarrows SH_{S^1,\bb G}^\Sigma(\Re):U.$$
We refer the interested reader to~\cite{H,J} for further details.

\section{$K$-motives and $(S^1,\bb G)$-bispectra}

We prove in this section that the triangulated category of
$K$-motives is fully faithfully embedded into the stable homotopy
category of $(S^1,\bb G)$-bispectra $SH_{S^1,\bb G}(\Re)$. In
particular, the triangulated category $kk(\Re)$ of
Corti\~nas--Thom~\cite{CT} is fully faithfully embedded into
$SH_{S^1,\bb G}(\Re)$ by means of a contravariant functor. As an
application we construct an explicit fibrant $(S^1,\bb
G)$-bispectrum representing homotopy $K$-theory in the sense of
Weibel~\cite{W1}.

Throughout this section we assume that $\Re$ is a small tensor
closed, $T$-, $\Gamma$- and $\tau_0$-closed admissible category of
$k$-algebras. It follows that $\sigma A,\Sigma A,M_\infty A\in\Re$
for all $A\in\Re$.

Let $Sp^\Sigma_{\infty,\bb G}(\Re)$ denote the model category of
(usual, non-symmetric) $\bb G$-spectra in $Sp^\Sigma_\infty(\Re)$.
Using Hovey's notation~\cite{H}, $Sp^\Sigma_{\infty,\bb
G}(\Re)=Sp^{\bb N}(Sp^\Sigma_\infty(\Re),-\otimes\bb G)$.

\begin{prop}\label{kmotinter}
The functor
   $$-\otimes\bb G:Sp_\infty^\Sigma(\Re)\to Sp_\infty^\Sigma(\Re)$$
and the canonical functor
   $$F_{0,\bb G}=\Sigma_{\bb G}^\infty:Sp_\infty^\Sigma(\Re)\to Sp^\Sigma_{\infty,\bb G}(\Re)$$
are left Quillen equivalences.
\end{prop}

\begin{proof}
We first observe that $-\otimes\bb G$ is a left Quillen equivalence
on $Sp^\Sigma_{\infty}(\Re)$ if and only if so is
$-\otimes\Sigma^\infty\bb G$. By~\cite[section~4]{CT} there is an
extension
   $$M_\infty k\rightarrowtail\tau_0\twoheadrightarrow\sigma.$$
It follows from~\cite[7.3.2]{CT} that $\Sigma^\infty(r(\tau_0))=0$
in $DK(\Re)$, and hence $\Sigma^\infty(r(\tau_0))$ is weakly
equivalent to zero in $Sp_\infty^\Sigma(\Re)$.

The extension above yields therefore a zig-zag of weak equivalences
between cofibrant objects in $Sp_\infty^\Sigma(\Re)$ from
$\Sigma^\infty(r(M_\infty k))$ to $\Sigma^\infty\bb G\wedge S^1$.
Since $\Sigma^\infty(r(M_\infty k))$ is weakly equivalent to the
monoidal unit $\Sigma^\infty(r(k))$, we see that
$\Sigma^\infty(r(k))$ is zig-zag weakly equivalent to
$(\Sigma^\infty\bb G)\wedge S^1$ in the category of cofibrant
objects in $Sp^\Sigma_\infty(\Re)$.

Since $\Sigma^\infty(r(k))$ is a monoidal unit in
$Sp^\Sigma_\infty(\Re)$, then $-\otimes\Sigma^\infty(r(k))$ is a
left Quillen equivalence on $Sp^\Sigma_\infty(\Re)$, and hence so is
$-\otimes((\Sigma^\infty\bb G)\wedge S^1))$. But $-\wedge S^1$ is a
left Quillen equivalence on $Sp^\Sigma_\infty(\Re)$. Therefore
$-\otimes\Sigma^\infty\bb G$ is a left Quillen equivalence
by~\cite[1.3.15]{Hov}.

The fact that the canonical functor
   $$F_{0,\bb G}:Sp_\infty^\Sigma(\Re)\to Sp^\Sigma_{\infty,\bb G}(\Re)$$
is a left Quillen equivalence now follows from~\cite[5.1]{H}.
\end{proof}

Denote the homotopy category of $Sp_{\infty,\bb G}^\Sigma(\Re)$ by
$SH^{\Sigma,\infty}_{S^1,\bb G}(\Re)$.

\begin{cor}\label{pele}
The canonical functor
   $$F_{0,\bb G}=\Sigma_{\bb G}^\infty:DK(\Re)\to SH^{\Sigma,\infty}_{S^1,\bb G}(\Re)$$
is an equivalence of triangulated categories.
\end{cor}

Recall that $Sp_\infty^\Sigma(\Re)$ is the Bousfield localization of
$Sp^\Sigma(\Re)$ with respect to
   $$\{F_s(r(M_\infty A))\to F_s(rA)\mid A\in\Re,s\geq 0\}.$$
It follows that the induced triangulated functor is fully faithful
   $$DK(\Re)\to\shs^\Sigma(\Re).$$

In a similar fashion, $Sp_{\infty,\bb G}^\Sigma(\Re)$ can be
obtained from $Sp^\Sigma_{\bb G}(\Re)$ by Bousfield localization
with respect to
   $$\{F_{k,\bb G}(F_s(r(M_\infty A)))\to F_{k,\bb G}(F_s(rA))\mid A\in\Re,k,s\geq 0\}.$$
We summarize all of this together with Proposition~\ref{kmotinter}
as follows.

\begin{thm}\label{arshavin}
There is an adjoint pair of triangulated functors
   $$\Phi:SH_{S^1,\bb G}^\Sigma(\Re)\leftrightarrows DK(\Re):\Psi$$
such that $\Psi$ is fully faithful. Moreover, $\cc T=\kr\Phi$ is the
localizing subcategory of $SH_{S^1,\bb G}^\Sigma(\Re)$ generated by
the compact objects
   $$\{\cone(F_{k,\bb G}(F_s(r(M_\infty A)))\to F_{k,\bb G}(F_s(rA)))\mid A\in\Re\}$$
and $DK(\Re)$ is triangle equivalent to $SH^\Sigma_{S^1,\bb
G}(\Re)/\cc T$.
\end{thm}

\begin{cor}\label{zidan}
There is a contravariant fully faithful triangulated functor
   $$kk(\Re)\to SH_{S^1,\bb G}(\Re).$$
\end{cor}

\begin{proof}
This follows from Theorems~\ref{kktop} and~\ref{arshavin}.
\end{proof}

Let $Sp_{\infty,\bb G}(\Re)$ denote the model category of $\bb
G$-spectra in $Sp_\infty(\Re)$. Using Hovey's notation~\cite{H},
$Sp_{\infty,\bb G}(\Re)=Sp^{\bb N}(Sp_\infty(\Re),-\otimes\bb G)$.
As above, there is a Quillen equivalence
   $$V:Sp_{\infty,\bb G}(\Re)\leftrightarrows Sp^\Sigma_{\infty,\bb G}(\Re):U,$$
where $U$ is the forgetful functor. It induces an equivalence of
triangulated categories
   $$V:SH^\infty_{S^1,\bb G}(\Re)\leftrightarrows SH^{\Sigma,\infty}_{S^1,\bb G}(\Re):U,$$
where the left hand side is the homotopy category of $Sp_{\infty,\bb
G}(\Re)$.

Given $A\in\Re$, consider a $(S^1,\bb G)$-bispectrum
$\bb{KG}^{st}(A,-)$ which we define at each $B\in\Re$ as
   $$\colim_n(\bb{KG}(A,B)\to\bb{KG}(A,M_\infty k\otimes B))\to\bb{KG}(A,M_\infty^2k\otimes B)\to\cdots)$$
It can also be presented as the array
\begin{equation*}
\xymatrix{
&\vdots & \vdots & \vdots\\
\bb K^{st}(\sigma^2 A,B):&\cc K^{st}(\sigma^2 A,B) & \cc K^{st}(J\sigma^2 A,B) & \cc K^{st}(J^2\sigma^2 A,B) & \cdots\\
\bb K^{st}(\sigma A,B): &\cc K^{st}(\sigma A,B)& \cc K^{st}(J\sigma A,B) & \cc K^{st}(J^2\sigma A,B) & \cdots\\
\bb K^{st}(A,B):&\cc K^{st}(A,B) & \cc K^{st}(JA,B) & \cc
K^{st}(J^2A,B) & \cdots }
\end{equation*}

It follows from Theorem~\ref{mainresst} that the canonical map of
bispectra
   $$\Sigma_{\bb G}^\infty\Sigma^\infty rA\to\bb{KG}^{st}(A,-)$$
is a level weak equivalence in $Sp_{\infty,\bb G}(\Re)$. In fact we
can say more. We shall show below that $\bb{KG}^{st}(A,-)$ is a
fibrant bispectrum and this arrow is a fibrant replacement of
$\Sigma_{\bb G}^\infty\Sigma^\infty rA$ in $Sp_{\infty,\bb G}(\Re)$.
To this end we have to prove the Cancellation Theorem for the
$S^1$-spectrum $\bb{K}^{st}(A,-)$. The Cancellation Theorem for
$K$-theory of algebraic varieties was proved in~\cite{GP2}. It is
also reminiscent of the Cancellation Theorem for motivic cohomology
proved by Voevodsky~\cite{Voe3}.

\begin{thm}[Cancellation for $K$-theory]\label{novgor}
Each structure map of the bispectrum $\bb{KG}^{st}(A,-)$
   \begin{equation*}\label{len}
    \bb{K}^{st}(\sigma^n A,-)\to\Omega_{\bb G}\bb{K}^{st}(\sigma^{n+1}A,-),\quad n\geq 0,
   \end{equation*}
is a weak equivalence of fibrant $S^1$-spectra.
\end{thm}

\begin{proof}
It follows from Proposition~\ref{kmotinter} that the functor
   $$-\otimes\bb G:Sp_\infty(\Re)\to Sp_\infty(\Re)$$
is a left Quillen equivalence. It remains to apply
Theorem~\ref{mainresst}.
\end{proof}

\begin{cor}\label{ronie}
For any $A\in\Re$ the bispectrum $\bb{KG}^{st}(A,-)$ is fibrant in
$Sp_{\infty,\bb G}(\Re)$. Moreover, the canonical map of bispectra
   $$\Sigma_{\bb G}^\infty\Sigma^\infty rA\to\bb{KG}^{st}(A,-)$$
is a fibrant resolution for $\Sigma_{\bb G}^\infty\Sigma^\infty rA$
in $Sp_{\infty,\bb G}(\Re)$.
\end{cor}

The following result says that the bispectrum $\bb{KG}^{st}(A,-)$ is
$(2,1)$-periodic and represents stable algebraic Kasparov $K$-theory
(cf.~\cite[6.8-6.9]{VoeICM}).

\begin{thm}\label{maradona}
For any algebras $A,B\in\Re$ and any integers $p,q$ there is an
isomorphism of abelian groups
   $$\pi_{p,q}(\bb{KG}^{st}(A,B))\cong\Hom_{SH_{S^1,\bb G}(\Re)}
     (\Sigma_{\bb G}^\infty\Sigma^\infty rB\otimes S^{p-q}\otimes\bb G^q,\bb{KG}^{st}(A,-))\cong\bb K^{st}_{p-2q}(A,B).$$
In particular,
   $$\pi_{p,q}(\bb{KG}^{st}(A,B))\cong\pi_{p+2,q+1}(\bb{KG}^{st}(A,B)).$$
\end{thm}

\begin{proof}
By Corollary~\ref{ronie} the bispectrum $\bb{KG}^{st}(A,-)$ is a
fibrant replacement for $\Sigma_{\bb G}^\infty\Sigma^\infty rA$ in
$Sp_{\infty,\bb G}(\Re)$. Therefore,
   $$\pi_{p,q}(\bb{KG}^{st}(A,B))\cong\Hom_{SH_{S^1,\bb G}^\infty(\Re)}(\Sigma_{\bb G}^\infty\Sigma^\infty rB\otimes S^{p-q}\otimes\bb G^q,
     \Sigma_{\bb G}^\infty\Sigma^\infty rA).$$
Corollary~\ref{pele} implies that the right hand side is isomorphic
to $DK(\Re)(M_K(B)\otimes S^{p-q}\otimes\bb G^q,M_K(A))$. On the
other hand,
   $$DK(\Re)(M_K(B)\otimes S^{p-q}\otimes\bb G^q,M_K(A))\cong
     DK(\Re)(M_K(B)\otimes S^{p-2q}\otimes S^{q}\otimes\bb G^q,M_K(A)).$$
The proof of Proposition~\ref{kmotinter} implies
$\Sigma^\infty(S^1\otimes\bb G)$ is isomorphic to the monoidal unit.
Therefore,
   $$DK(\Re)(M_K(B)\otimes S^{p-2q}\otimes S^{q}\otimes\bb G^q,M_K(A))\cong DK(\Re)(M_K(B)[p-2q],M_K(A)).$$
Our statement now follows from Theorem~\ref{kmotives}.
\end{proof}

The next statement says that the bispectrum $\bb{KG}^{st}(k,B)$
gives a model for homotopy $K$-theory in the sense of
Weibel~\cite{W1} (cf.~\cite[6.9]{VoeICM}).

\begin{cor}\label{ronaldo}
For any algebra $B\in\Re$ and any integers $p,q$ there is an
isomorphism
   $$\pi_{p,q}(\bb{KG}^{st}(k,B))\cong KH_{p-2q}(B).$$
\end{cor}

\begin{proof}
This follows from the preceding theorem and~\cite[9.11]{Gark}.
\end{proof}

\end{document}